\documentclass[11pt]{article}
\usepackage{amsmath,amsthm,amssymb}
\usepackage{color}

\usepackage[hyperfootnotes=false]{hyperref}

\usepackage{titlesec}
\titleformat{\paragraph}
{\normalfont\normalsize\bfseries}{\theparagraph}{1em}{}
\titlespacing*{\paragraph}
{0pt}{3.25ex plus 1ex minus .2ex}{1.5ex plus .2ex}

\def\titlerunning#1{\gdef\titrun{#1}}
\makeatletter
\def\author#1{\gdef\autrun{\def\and{\unskip, }#1}\gdef\@author{#1}}
\def\address#1{{\def\and{\\\hspace*{18pt}}\renewcommand{\thefootnote}{}%
\footnote {#1}}%
\markboth{\autrun}{\titrun}}
\makeatother
\def\email#1{\hspace*{4pt}{\em e-mail}: #1}
\def\MSC#1{{\renewcommand{\thefootnote}{}%
\footnote{\emph{Mathematics Subject Classification (2020):} #1}}}
\def\keywords#1{\par\medskip
\noindent\textbf{Keywords:} #1}

\newtheorem{theorem}{Theorem}[section]
\newtheorem{prop}[theorem]{Proposition}
\newtheorem{cor}[theorem]{Corollary}
\newtheorem{lemma}[theorem]{Lemma}
\newtheorem{prob}[theorem]{Problem}

\newtheorem{defin}[theorem]{Definition}
\newtheorem{result}[theorem]{Result}

\theoremstyle{definition}
\newtheorem{construction}[theorem]{Construction}
\newtheorem{remark}[theorem]{Remark}

\numberwithin{equation}{section}

\def\u{{\boldsymbol u}}
\def\x{{\boldsymbol x}}
\def\y{{\boldsymbol y}}
\def\j{{\boldsymbol j}}
\def\0{\mathbf 0}

\def\P{\mathbf P}

\def\cL{\mathcal L}

\def\cA{\mathcal A}
\def\cC{\mathcal C}
\def\cB{\mathcal B}
\def\cE{\mathcal E}
\def\cF{\mathcal F}
\def\cG{\mathcal G}
\def\cH{\mathcal H}

\def\cM{\mathcal M}
\def\cK{\mathcal K}
\def\cO{\mathcal O}
\def\cP{\mathcal P}
\def\cX{\mathcal X}
\def\cD{\mathcal D}

\def\cR{\mathcal R}
\def\cS{\mathcal S}
\def\cW{\mathcal W}
\def\cQ{\mathcal Q}
\def\cZ{\mathcal Z}

\def\PG{{\rm PG}}
\def\GF{{\rm GF}}

\def\PGL{{\rm PGL}}
\def\PSL{{\rm PSL}}
\def\GaL{{\rm \Gamma L}}
\def\GL{{\rm GL}}
\def\SL{{\rm SL}}

\def\PGO{{\rm PGO}}

\def\PSp{{\rm PSp}}

\def\Sp{{\rm Sp}}

\def\PU{{\rm PU}}
\def\PSU{{\rm PSU}}

\def\U{{\rm U}}
\def\SU{{\rm SU}}

\def\diag{{\rm diag}}

\def\R{{\mathbb R}}

\begin{document}


\baselineskip=16pt


\titlerunning{}

\title{On regular systems of finite classical polar spaces}

\author{Antonio Cossidente \and Giuseppe Marino \and Francesco Pavese \and Valentino Smaldore}

\date{}

\maketitle

\address{A. Cossidente: Dipartimento di Matematica, Informatica ed Economia, Universit{\`a} degli Studi della Basilicata, Contrada Macchia Romana, 85100, Potenza, Italy; \email{antonio.cossidente@unibas.it}
\and
G. Marino: Dipartimento di Matematica e Applicazioni ``Renato Caccioppoli'', Universit{\`a} degli Studi di Napoli ``Federico II'', Complesso Universitario di Monte Sant'Angelo, Cupa Nuova Cintia 21, 80126, Napoli, Italy; \email{giuseppe.marino@unina.it}
\and
F. Pavese: Dipartimento di Meccanica, Matematica e Management, Politecnico di Bari, Via Orabona 4, 70125 Bari, Italy; \email{francesco.pavese@poliba.it}
\and
V. Smaldore: Dipartimento di Matematica, Informatica ed Economia, Universit{\`a} degli Studi della Basilicata, Contrada Macchia Romana, 85100, Potenza, Italy; \email{valentino.smaldore@unibas.it}
}

\bigskip

\MSC{Primary 51E20; Secondary 05B25}


\begin{abstract}
Let $\cP$ be a finite classical polar space of rank $d$. An $m$--regular system with respect to $(k-1)$--dimensional projective spaces of $\cP$, $1 \le k \le d-1$, is a set $\cal R$ of generators of $\cP$ with the property that every $(k-1)$--dimensional projective space of $\cP$ lies on exactly $m$ generators of $\cR$. Regular systems of polar spaces are investigated. Some non--existence results about certain $1$--regular systems of polar spaces with low rank are proved and a procedure to obtain $m'$--regular systems from a given $m$--regular system is described. Finally, three different construction methods of regular systems w.r.t. points of various polar spaces are discussed.

\end{abstract}

\section{Introduction}

Let $\cP$ be a finite classical polar space, i.e., $\cP$ arises from a vector space of finite dimension over a finite field equipped with a non-degenerate reflexive sesquilinear or quadratic form. Hence $\cP$ is a member of one of the following classes: a symplectic space $\cW(2n+1,q)$, a parabolic quadric $\cQ(2n,q)$, a hyperbolic quadric $\cQ^+(2n+1,q)$, an elliptic quadric $\cQ^-(2n+1,q)$ or a Hermitian variety $\cH(n,q)$ ($q$ a square). A projective subspace of maximal dimension contained in $\cP$ is called a {\em generator} of $\cP$. The vector dimension of a generator of $\cP$ is called the {\em rank} of $\cP$. Here and in the sequel we will use the term polar space to refer to a finite classical polar space and $\cP_{d, e}$ will denote a polar space of rank $d \ge 2$ as follows:
$$
\begin{tabular}{c|c|c|c|c|c|c}
$\cP_{d, e}$ & $\cQ^+(2d-1,q)$ & $\cH(2d-1,q)$ & $\cW(2d-1, q)$ & $\cQ(2d, q)$ & $\cH(2d, q)$ & $\cQ^-(2d+1,q)$ \\
\hline
$e$  &  $0$ & $1/2$ & $1$ & $1$ & $3/2$ & $2$ 
\end{tabular}
$$
The term {\em $(k - 1)$--space of $\cP_{d, e}$} will be used to denote a totally isotropic or totally singular $(k - 1)$--dimensional projective space of $\cP_{d, e}$ and $\cM_{\cP_{d, e}}$ will denote the set of generators of $\cP_{d, e}$. The number of $(k-1)$--spaces of $\cP_{d, e}$ is given by
$$
\genfrac{[}{]}{0pt}{}{d}{k}_q \prod_{i = 1}^{k} \left(q^{d + e - i} + 1\right).
$$
The number $e$ is given above for each finite polar space.
For further details on finite classical polar spaces we refer to \cite{BCN, H, H1, HT, taylor}. The notion of regular system of a polar space dates back to $1965$ and was introduced for the first time by B. Segre in the context of Hermitian varieties \cite{Segre1}, see also \cite[Section 19.3]{H}.

\begin{defin}
An {\bf $m$--regular system} (or simply {\bf regular system}) with respect to $(k-1)$--spaces of $\cP_{d, e}$, $1 \le k \le d-1$, is a set $\cal R$ of generators of $\cP_{d, e}$ with the property that every $(k-1)$--space of $\cP_{d, e}$ lies on exactly $m$ generators of $\cal R$, $0 \le m \le |\cM_{\cP_{d-k, e}}|$, where $|\cM_{\cP_{d-k, e}}|$ is the number of generators of $\cP_{d, e}$ passing through one of its $(k-1)$--spaces.
\end{defin}

The concept of regular system generalizes the well known notion of {\em spread} \cite[Chapter 7]{HT}. A regular system with respect to $(k - 1)$--spaces of $\cP_{d, e}$ having the same size as its complement (in $\cM_{\cP_{d, e}}$) is said to be a {\bf hemisystem} with respect to $(k - 1)$--spaces of $\cP_{d, e}$ (of course here $|\cM_{\cP_{d, e}}|$ must be even). From \cite{BLL, Segre1, Van1}, a regular system w.r.t. $(d-2)$--spaces of $\cP_{d, e}$, $d \ge 3$ and $e \le 1$, or $(d, e) = (2, 1/2)$ is a hemisystem. Not much is known about regular systems w.r.t. $(k-1)$--spaces of a polar space. If $k \ge 2$ and the polar space is not a hyperbolic quadric only sporadic examples are known and they are hemisystems w.r.t. lines of $\cQ(6, q)$, $q\in\{3, 5\}$, see \cite{BLL}. If $k = 1$ and the rank $d$ of the polar space is at least three, then some constructions are known on the parabolic quadric $\cQ(2d, q)$ \cite{CPa, LN} and on the Hermitian variey $\cH(2d-1, q)$ \cite{LB, CP1}. Finally since polar spaces of rank two are generalized quadrangles, several examples of regular systems arise from $m$--ovoids by using duality, see \cite{BDS, BGR, BLMX, BLP, BW, CCEM, CPa, CP, FMX, FT, KNS, PW, Pa}. However, in the case of generalized quadrangles, many questions are still unsolved.

In this paper we deal with regular systems of polar spaces. In Section \ref{sec_1} preliminary notions are introduced along with some non--existence results about certain $1$--regular systems of polar spaces having low rank. In Section \ref{sec_2} we describe how in some cases it is possible to obtain $m'$--regular systems starting from a given $m$--regular system. In the remaining part of the paper we are mainly concerned with constructions of regular systems w.r.t. points of various polar spaces. In particular three different methods are presented: by partitioning the generators of an elliptic quadric $\cQ^-(2d+1, q)$ into generators of hyperbolic quadrics $\cQ^+(2d-1, q)$ embedded in it (Section~\ref{sec_3}); regular systems w.r.t. points of a polar space $\cP$ obtained by means of a $k$--system of $\cP$ (Section~\ref{sec_4}); regular systems w.r.t. points arising from field reduction (Section~\ref{sec_5}).

\section{Preliminaries and first results}\label{sec_1}

We summarize some straightforward properties regarding regular systems of a polar space $\cP_{d, e}$.

\begin{lemma}\label{properties}
Let $\cA$ and $\cB$ be an $m$--regular system and an $m'$--regular system w.r.t. $(k-1)$--spaces of $\cP_{d, e}$, respectively. Then:
\begin{itemize}
\item[$i)$] $| \cA | = m \prod_{i = 1}^{k} \left(q^{d + e - i} + 1\right)$; 
\item[$ii)$] a regular system of $\cP_{d, e}$ w.r.t. $(k-1)$--spaces is also a regular system w.r.t. $(k'-1)$--spaces, $1 \le k' \le k-1$;
\item[$iii)$] if $\cA \subseteq \cB$, then $\cB \setminus \cA$ is $(m'-m)$--regular system w.r.t. $(k-1)$--spaces of $\cP_{d, e}$;
\item[$iv)$] if $\cA$ and $\cB$ are disjoint, then $\cA \cup \cB$ is $(m+m')$--regular system w.r.t. $(k-1)$--spaces of $\cP_{d, e}$;
\item[$v)$] the empty set and $\cM_{\cP_{d, e}}$ are trivial examples of $(k-1)$--regular systems of $\cP_{d, e}$, for each $1\leq k\leq d-1$.
\end{itemize}
\end{lemma}

A $d$--{\em class association scheme} is a set $X$, together with $d+1$ symmetric relations $R_i$ on $X$ such that $\{R_0,\dots,R_d\}$ partitions $X\times X$, the identity relation on $X\times X$ is $R_0$, and for any $(x,y)\in R_k$, the numbers $p^k_{i j} =|\{z\in X \;\; | \;\; (x,z)\in R_i \;{\rm and}\; (z,y)\in R_j\}|$ depend only on $i, j, k$ (and not on $x$ and $y$).  Association schemes were first defined in 1952 by Bose and Shimamoto \cite{BS}.  A strongly regular graph gives rise to a $2$-class association scheme by defining $(x, y)\in  R_1$ if $x$ is adjacent to $y$, and $(x, y)\in R_2$ if $x$ and $y$ are distinct and non--adjacent. Any $2$-class association scheme can be viewed as coming from a strongly regular graph in this way. Similarly, a distance regular graph with diameter $d$ gives a $d$--class association scheme by defining $(x, y) \in R_i$ if $d(x, y) = i$.  

%
%


Let $\cD^i_{\cP_{d, e}}$ be the {\em $i$--th distance graph of $\cP_{d, e}$}, that is the graph having as vertex the members of $\cM_{\cP_{d, e}}$ and two vertices $x, y \in \cM_{\cP_{d, e}}$ are adjacent whenever $x \cap y$ is a $(d-i-1)$--space of $\cP_{d, e}$. $0 \le i \le d$. The {\em $i$--th} distance graph $\cD^i_{\cP_{d, e}}$ is regular. If $i = 1$, then $\cD^1_{\cP_{d, e}}$ is called {\em dual polar graph of $\cP_{d, e}$} and it is symply denoted by $\cD_{\cP_{d, e}}$. The graph $\cD_{\cP_{d, e}}$ is a distance--regular graph with diameter $d$, and hence it gives a $d$--class association scheme. Two vertices of  $\cD_{\cP_{d,e}}$ are adjacent if the corresponding generators meet in a $(d-2)$--space. For more information on $\cD_{\cP_{d, e}}$, the reader is referred to \cite[Section 9.4]{BCN}.


For a matrix $M$, we denote its column span by $Im(M)$ and its transpose by $M^t$. Let $I$ be the identity matrix, $J$ be the all ones matrix and $\j$ the all-one vector of suitable length. Fix an ordering on $\cM_{\cP_{d, e}}$. This provides an ordered basis for the vector space $V = \R^{\cM_{\cP_{d, e}}}$. Note that $V$ has dimension $n$, where $n = |\cM_{\cP_{d, e}}|$. For a subset $S \subseteq \cM_{\cP_{d, e}}$, let $\chi_{S}$ be the characteristic column vector of $S$. Let $A_i$ be the adjacency matrix of the {\em $i$--th} distance graph $\cD^i_{\cP_{d, e}}$ of $\cP_{d, e}$ for $0 \le i \le d$, that is, $A_i$ will have a $1$ in the $(j, k)$ position if $d(x_j, x_k) = i$ and $0$ otherwise. This gives a set of $d + 1$ symmetric matrices $A_0, \dots, A_d$, with $A_0 = I$, $\sum_{i = 0}^{d} A_i = J$, and $A_i A_ j =\sum_{k = 0}^{d} p^k_{i j} A_k$. The $\R$--vector space generated by the matrices $A_0, \dots, A_d$ is closed under matrix multiplication and so it is a $(d + 1)$--dimensional commutative $\R$--algebra, also known as the {\em Bose--Mesner algebra} of the association scheme. 

The Bose--Mesner algebra admits a basis $\{E_0, \dots, E_d\}$ consisting of minimal symmetric idempotents. Here $E_0 = J/n$. Moreover, $E_i E_j = \delta_{i j} E_i$ and every minimal idempotent $E_j$ is such that $A_1 E_j = \lambda_j E_j$, where $\lambda_j$ is an eigenvalue of $A_1$. The subspaces $V_0, \dots, V_d$, with $V_i = Im(E_i)$, form an orthogonal decomposition of $V$, i.e., $V = V_0 \perp \dots \perp V_d$. Let $C_k$ be the incidence matrix between the $(k-1)$--spaces of $\cP_{d, e}$ and the generators of $\cP_{d, e}$, $1 \le k \le d$. This means that the rows are indexed by the $(k-1)$--spaces, the columns by the generators of the polar space and $(C_k)_{y x} = 1$ if $y \subseteq x$ and $(C_k)_{y x} = 0$ if $y \not\subseteq x$. Now up to a reordering of the indices we have that $Im(C_k^t) = V_0 \perp \dots \perp V_k$, where $V_0=Im(C_0^t)$  and $V_k = Im(C_k^t) \cap Ker(C_{k-1})$. For more results on this topic the interested reader is referred to \cite{BCN, Delsarte1, Delsarte2, Van, Van2}.

\bigskip
The {\em dual degree set} of a set $S$ of vertices of $\cD_{\cP_{d, e}}$ is the set of non--zero indices $1 \le i \le d$ such that $E_i \chi_{S} \ne 0$ or equivalently $\chi_{S} \notin V_i^\perp$. Two sets of vectors $S_1$ and $S_2$ of $\cD_{\cP_{d, e}}$ are called {\em design--orthogonal} when their dual degree sets are disjoint. 

If $|\{ 1 \le i \le d \;\; | \;\; E_i \chi_{S} \ne 0\} \cap \{1, \dots, k\}| = 0$, then $S$ is said to be a {\em $k$--design}, whereas if $\{ 1 \le i \le d \;\; | \;\; E_i \chi_{S} \ne 0\} \subseteq \{1, \dots, k\}$ then $S$ is called a {\em $k$--antidesign}. Clearly $k$--designs and $k$--antidesigns are design--orthogonal subsets.

The following result has been proved in \cite[Theorem 4.4.1]{Van}, see also \cite{Delsarte2}.

\begin{prop}\label{design}
	A set $\cR \subset \cM_{\cP_{d, e}}$ is an $m$--regular system with respect to $(k-1)$--spaces of $\cP_{d, e}$, $1 \le k \le d-1$, if and only if $\cR$ is a $k$--design. 
\end{prop} 

The following result on the intersection size between design--othogonal subsets can be found in \cite{Delsarte3, Roos}.

\begin{result}
If $S_1$ and $S_2$ are design--orthogonal subsets of an association scheme $X$, then 
\begin{equation}\label{orthogonal}
|S_1 \cap S_2| = \frac{|S_1| |S_2|}{|X|}. 
\end{equation}
\end{result}

The eigenvalues of the {\em $i$--th} distance graph $\cD^i_{\cP_{d, e}}$ of $\cP_{d, e}$ are those of its adjacency matrix $A_i$. A {\em coclique} of $\cD^i_{\cP_{d, e}}$ is a set of pairwise nonadjacent vertices. Hence a coclique of $\cD^i_{\cP_{d, e}}$ is a set of generators of $\cP_{d, e}$  pairwise not intersecting in a $(d-i-1)$--space. The independence number $\alpha(\cD^i_{\cP_{d, e}})$ is the size of the largest coclique of $\cD^i_{\cP_{d, e}}$. The graph $\cD^i_{\cP_{d, e}}$ has at most $d+1$ distinct eigenvalues \cite{Eisfeld}, \cite[Theorem 4.3.6]{Van}. Let $k_i$ and $\lambda_{i}$ denote the largest and the smallest eigenvalue of $\cD^i_{\cP_{d, e}}$, respectively; hence $k_i$ is the valency of $\cD^i_{\cP_{d, e}}$. The following result is due to Hoffman, see \cite[Theorem 3.5.2]{BH}.

\begin{lemma}\label{hoffman}
$\alpha(\cD^i_{\cP_{d, e}}) \le -\frac{|\cM_{\cP_{d, e}}| \lambda_i}{k_i - \lambda_i}$.
\end{lemma}

Let $\cR$ be a $1$--regular system w.r.t. $(k - 1)$--spaces of $\cP_{d, e}$. Then every $(k-1)$--space of $\cP_{d, e}$ is contained in exactly one member of $\cR$ and therefore two distinct generators of $\cR$ intersect in at most a $(k-2)$--space of $\cP_{d, e}$. In particular $\cR$ is a coclique of the graph $\cD^{d-k}_{\cP_{d, e}}$ and hence 
\begin{equation}\label{inequality}
|\cR| \le -\frac{|\cM_{\cP_{d, e}}| \lambda_i}{k_i - \lambda_i}.
\end{equation}
In some cases Equation \eqref{inequality} yields a contradiction. We study the cases when $\cR$ is a $1$--regular system of a polar space with rank $4$ or $5$.
\begin{theorem}
The polar spaces $\cQ^+(7, q)$, $\cH(7, q)$, $\cW(7, q)$, $\cQ(8, q)$, $\cH(8, q)$, $\cQ^-(9, q)$ do not have a $1$--regular system w.r.t. lines. The polar spaces $\cQ^+(9, q)$, $\cH(9, q)$, $\cW(9, q)$, $\cQ(10, q)$, $\cH(10, q)$, $\cQ^-(11, q)$ do not have a $1$--regular system w.r.t. planes. 
\end{theorem}
\begin{proof}
From \cite[Theorem 4.3.6]{Van}, the eigenvalues of $\cD^{2}_{\cP_{4, e}}$ are $q^{2e+1}(q^2+1)(q^2+q+1), (q^{2e+1} - q^e)(q^2+q+1), -q^e(q+1)^2+q^{2e+1}+q, -(q^e-q)(q^2+q+1), q(q^2+1)(q^2+q+1)$. Let $\cR$ be a $1$--regular system w.r.t. lines of a polar space $\cP_{4, e}$. By Equation \eqref{inequality} we have that 
\begin{align*}
& (q^2+1)(q^3+1) = |\cR| \le 2(q^3+1), & \mbox{ for } \cQ^+(7, q), \\ 
& \left(q^\frac{5}{2}+1\right)\left(q^\frac{7}{2}+1\right) = |\cR| \le \frac{ \left(q^\frac{1}{2}(q+1)^2-q^2-q\right) \prod_{i = 1}^{4} \left( q^{\frac{9-2i}{2}} + 1 \right)}{q^2(q^2+1)(q^2+q+1) + q^\frac{1}{2}(q+1)^2-q^2-q} & \mbox{ for } \cH(7, q), \\
& (q^3+1)(q^4+1) = |\cR| \le \frac{2(q+1)(q^2+1)(q^3+1)(q^4+1)}{q(q^2+1)(q^2+q+1) + 2} & \mbox{ for } \cW(7, q), \cQ(8, q), \\
& \left( q^\frac{7}{2} + 1\right)\left(q^\frac{9}{2} + 1\right) \le \frac{\left( q^\frac{1}{2} - 1 \right) \left( q^\frac{3}{2} + 1 \right) \left( q^\frac{5}{2} + 1 \right) \left( q^\frac{7}{2} + 1 \right) \left( q^\frac{9}{2} + 1 \right)}{q^3(q^2+1) + q^\frac{1}{2} - 1} & \mbox{ for } \cH(8, q), \\
& (q^4+1)(q^5+1) = |\cR| \le \frac{(q-1)(q^2+1)(q^3+1)(q^4+1)(q^5+1)}{q^4(q^2+1)+q-1}  & \mbox{ for } \cQ^-(9, q).
\end{align*}
In all cases we get a contradiction. 

Similarly the largest and the smallest eigenvalue of $\cD^2_{\cP_{5, e}}$ are $q^{2e+1}(q^2+1)(q^5-1)/(q-1)$ and $-q^e(q+1)(q^2+q+1)+q^{2e+1}+q(q^2+q+1)$. Hence if $\cR$ is a $1$--regular system w.r.t. planes of $\cP_{5, e}$, it follows that  
\begin{align*}
& \prod_{i = 1}^{3} \left( q^{5-i} + 1 \right) = |\cR| \le \frac{2(q+1)(q^2+1)(q^4+1)}{q^2+q+1} & \mbox{ for } \cQ^+(9, q), \\
& \prod_{i = 1}^{3} \left( q^{\frac{11-2i}{2}} + 1 \right) = |\cR| \le  \frac{ \left(q^\frac{1}{2}(q+1)(q^2+q+1) -q (q+1)^2 \right) \prod_{i = 1}^{5} \left( q^{\frac{11-2i}{2}} + 1 \right)}{\frac{q^2(q^2+1) (q^5-1)}{q-1} + q^\frac{1}{2}(q+1)(q^2+q+1) -q (q+1)^2} & \mbox{ for } \cH(9, q), \\
& \prod_{i = 1}^{3}(q^{6-i}+1) = |\cR| \le \frac{(q+1)(q^2+1)(q^4+1)(q^5+1)}{q^2+q+1} & \mbox{ for } \cW(7, q), \cQ(8, q), \\
& \prod_{i = 1}^{3} \left( q^{\frac{13-2i}{2}} + 1 \right) = |\cR| \le  \frac{ \left(q^\frac{3}{2}(q+1)(q^2+q+1) -q (q+1)(q^2+1) \right) \prod_{i = 1}^{5} \left( q^{\frac{13-2i}{2}} + 1 \right)}{\frac{q^4(q^2+1) (q^5-1)}{q-1} + q^\frac{3}{2}(q+1)(q^2+q+1) - q (q+1)(q^2+1)} & \mbox{ for } \cH(10, q), \\
& \prod_{i = 1}^{3} (q^{7-i}+1) = |\cR| \le \frac{(2q^4+q^3-q) \prod_{i = 1}^{5} (q^{7-i}+1)}{\frac{q^5(q^2+1)(q^5-1)}{q-1} + 2q^4+q^3-q} & \mbox{ for } \cQ^-(11, q).
\end{align*}
Again in all cases we get a contradiction.
\end{proof}

In order to obtain similar results for polar spaces of higher rank, the knowledge of the smallest eigenvalue of the {\em $i$--th} distance graph is needed.

\begin{prob}
Determine the smallest eigenvalue of the graph $\cD^i_{\cP_{d, e}}$.
\end{prob}



\section{Chains of regular systems}\label{sec_2}

Let $\cP_{d, e}$ be one of the following polar spaces: $\cQ(2d, q)$, $\cQ^-(2d+1, q)$, $\cH(2d, q)$, $d \ge 2$, and consider a projective subspace of the ambient projective space meeting $\cP_{d, e}$ in a cone, say $\cK$, having as vertex a $(k-2)$--space of $\cP_{d, e}$ and as base a $\cQ^+(2d-2k+1, q)$, $\cQ(2d-2k+2, q)$, $\cH(2d-2k+1, q)$, respectively, where $1 \le k \le d-1$. Let $\Gamma_{r}$ be the set of $(r-1)$--spaces of $\cP_{d, e}$ contained in $\cK$, $1 \le r \le d-1$, and denote by $\Omega_{r}$ the set of generators of $\cP_{d, e}$ meeting $\cK$ in exactly an $(r-1)$--space, $\max\{k, d-k\} \le r \le d$. Hence $\Omega_{d}$ is the set of generators of $\cP_{d, e}$ contained in $\cK$.

\begin{lemma}[{\cite[Theorem 2.4]{Van2}}]\label{cone1}
Let $\sigma$ be an $(r-1)$--space of $\cP_{d, e}$, $1 \le r \le d-1$, and let $\cS$ be the set of generators containing $\sigma$, then $\cS$ is an $r$--antidesign.
\end{lemma}
\begin{proof}
Note that $\chi_{\cS} = C_r^t \chi_{\sigma}$. Since $Im(C_r^t) = V_0 \perp \dots \perp V_r$, the result follows.
\end{proof}

\begin{lemma}\label{anti}
$\Omega_{d}$ is a $k$--antidesign.
\end{lemma}
\begin{proof}
Recall that $\cP_{d, e}$ is one of the following polar spaces: $\cQ(2d, q)$, $\cQ^-(2d+1, q)$, $\cH(2d, q)$, $d \ge 2$. First we see that the statement is true if $k = 1$. Indeed
$$
C_1^t \chi_{\Gamma_{1}} = \genfrac{[}{]}{0pt}{}{d}{1}_q \chi_{\Omega_{d}} + \genfrac{[}{]}{0pt}{}{d-1}{1}_q \chi_{\Omega_{d-1}} = \genfrac{[}{]}{0pt}{}{d}{1}_q \chi_{\Omega_{d}} + \genfrac{[}{]}{0pt}{}{d-1}{1}_q (\chi_{\cM_{\cP_{d, e}}} - \chi_{\Omega_{d}})
$$
and $C_1^t \j = \genfrac{[}{]}{0pt}{}{d}{1}_q \chi_{\cM_{\cP_{d, e}}}$. Hence $\chi_{\Omega_{d}} \in Im(C_1^t) = V_0 \perp V_1$. See also \cite[Theorem 5.1]{Van2}.

Assume that $\chi_{\Omega_{d}} \in Im(C_s^t) = V_0 \perp \dots \perp V_s$, whenever $\cK_s$ is a cone having as vertex an $(s-2)$--space of $\cP_{d, e}$, and as base a $\cQ^+(2d-2s+1, q)$, $\cQ(2d-2s+2, q)$, $\cH(2d-2s+1, q)$, respectively, $1 < s \le k-1$. Let $\Pi_r$ denote a subspace of the ambient space of projective dimension $2d-k+r$ if $\cP_{d, e} \in \{\cQ(2d, q), \cH(2d, q)\}$ or of projective dimension $2d-k+r+1$ if $\cP_{d, e} = \cQ^-(2d+1, q)$ and let $\cG_{\Pi_{r}}$ be the set of generators of $\cP_{d, e}$ contained in $\Pi_r$. If for some fixed $r$, with $1 \le r \le k-1$, we have that $\cK \subset \Pi_r$, then two possibilities occur: either $\Pi_r \cap \cP_{d, e}$ is a cone having as vertex a $(k - r - 1)$--space of $\cP_{d, e}$ and as base a $\cQ(2d - 2k + 2r, q)$, $\cQ^-(2d - 2k + 2r +1, q)$, $\cH(2d - 2k + 2r, q)$, respectively, or $\Pi_r \cap \cP_{d, e}$ is a cone having as vertex a $(k - r - 2)$--space of $\cP_{d, e}$ and as base a $\cQ^+(2d - 2k + 2r+1, q)$, $\cQ(2d-2k + 2r + 2, q)$, $\cH(2d - 2k + 2r + 1, q)$, respectively. By using Lemma~\ref{cone1} in the former case and our previous assumptions in the latter case, we have 
\begin{equation}\label{base}
\chi_{\cG_{\Pi_r}} \in Im(C_{k - r}^t) = V_0 \perp \dots \perp V_{k - r}.
\end{equation}
Also, observe that
\begin{equation} \label{case1}
C_k^t \chi_{\Gamma_k}= \genfrac{[}{]}{0pt}{}{d}{k}_q \chi_{\Omega_{d}} + \genfrac{[}{]}{0pt}{}{d-1}{k}_q \chi_{\Omega_{d-1}} + \dots + \genfrac{[}{]}{0pt}{}{d-k}{k}_q \chi_{\Omega_{d-k}}
\end{equation}
if $k \le \lfloor d/2 \rfloor$ and
\begin{align}
C_k^t \chi_{\Gamma_k} & = \genfrac{[}{]}{0pt}{}{d}{k}_q \chi_{\Omega_{d}} + \genfrac{[}{]}{0pt}{}{d-1}{k}_q \chi_{\Omega_{d-1}} + \dots + \genfrac{[}{]}{0pt}{}{k}{k}_q \chi_{\Omega_{k}} {\nonumber} \\ & = \genfrac{[}{]}{0pt}{}{d}{k}_q \chi_{\Omega_{d}} + \genfrac{[}{]}{0pt}{}{d-1}{k}_q \chi_{\Omega_{d-1}} + \dots + \genfrac{[}{]}{0pt}{}{d-(k-s')}{k}_q \chi_{\Omega_{d-(k-s')}} \label{case2}
\end{align}
if $k > \lfloor d/2 \rfloor$, with $k = d- (k-s')$, for some $s'$, where $1 \le s' \le k-1$.


Moreover
\begin{equation}\label{recursive1}
\chi_{\Omega_{d-r}} = \sum_{\cK \subset \Pi_r} \chi_{\cG_{\Pi_r}} - \sum_{i = 1}^r \genfrac{[}{]}{0pt}{}{k - r + i}{i}_q \chi_{\Omega_{d-r+i}}
\end{equation}
for $1 \le r \le k-1$ and
\begin{equation}\label{recursive2}
\chi_{\Omega_{d - k}} = \chi_{\cM_{\cP_{d, e}}} - \chi_{\Omega_d} - \chi_{\Omega_{d-1}} \dots - \chi_{\Omega_{d-k+1}}.
\end{equation}
Taking into account \eqref{recursive2} and applying recursively \eqref{recursive1}, we obtain that the hand--right side of Equation \eqref{case1} can be written as a linear combination of $\chi_{\Omega_{d}}, \chi_{\cM_{\cP_{d, e}}}, \chi_{\cG_{\Pi_1}}, \dots, \chi_{\cG_{\Pi_{k-1}}}$. Similarly, Equation \eqref{case2} can be written as a linear combination of $\chi_{\Omega_{d}}, \chi_{\cG_{\Pi_1}}, \dots, \chi_{\cG_{\Pi_{k-s'}}}$. The assertion follows from \eqref{base}.
\end{proof}



The next result shows that, by means of Lemma \ref{anti}, it is possible to construct $m'$--regular systems starting from a given $m$--regular system.

\begin{theorem}\label{chain}~
\begin{itemize}
\item[1)] If $\cQ^-(2d + 1, q)$ has an $m$--regular system w.r.t. $(k-1)$--spaces, then $\cQ(2d+2, q)$ has an $m(q+1)$--regular system w.r.t $(k-1)$--spaces and if $k \ge 2$, then $\cQ(2d, q)$ has an $m(q^{e-1}+1)$--regular system w.r.t. $(k-2)$--spaces.
\item[2)] If $\cQ(2d, q)$ has an $m$--regular system w.r.t. $(k-1)$--spaces, then $\cQ^+(2d+1, q)$ has a $2m$--regular system w.r.t. $(k-1)$--spaces and if $k \ge 2$, then $\cQ^+(2d-1, q)$ has an $m(q^{e-1}+1)$--regular system w.r.t. $(k-2)$--spaces.
\item[3)] If $\cH(2d, q)$ has an $m$--regular system w.r.t. $(k-1)$--spaces, then $\cH(2d+1, q)$ has an $m(q+1)$--regular system w.r.t. $(k-1)$--spaces and if $k \ge 2$, then $\cH(2d-1, q)$ has an $m(q^{e-1}+1)$--regular system w.r.t. $(k-2)$--spaces.
\end{itemize}
\end{theorem}
\begin{proof}
Let $\cP_{d, e}$ be one of the following polar spaces: $\cQ(2d, q)$, $\cQ^-(2d+1, q)$, $\cH(2d, q)$, $d \ge 2$, and let $\cR$ be an $m$--regular system of $\cP_{d, e}$ w.r.t. $(k-1)$--spaces. Embed $\cP_{d, e}$ in $\cP_{d+1, e-1}'$, where $\cP_{d+1, e-1}'$ is $\cQ^+(2d+1, q)$, $\cQ(2d+2, q)$, $\cH(2d+1, q)$, respectively. Notice that a generator of $\cP_{d+1, e-1}'$ contains exactly one generator of $\cP_{d, e}$. On the other hand, through a generator of $\cP_{d, e}$, there pass $q^{e-1}+1$ generators of $\cP_{d+1, e-1}'$. Let $\cR'$ be the set of generators of $\cP_{d+1, e-1}'$ containing a member of $\cR$. Let $\sigma$ be a $(k-1)$--space of $\cP_{d+1, e-1}'$ and let $\cC$ be the set of points of $\cP_{d+1, e-1}'$ belonging to at least a generator of $\cP_{d+1, e-1}'$ through $\sigma$. If $\sigma$ is contained in $\cP_{d, e}$, then through $\sigma$ there pass $m (q^{e-1}+1)$ elements of $\cR'$. If $\sigma$ meets $\cP_{d, e}$ in a $(k-2)$--space, then $\cK = \cC \cap \cP_{d, e}$ is a cone having as vertex $\sigma \cap \cP_{d, e}$ and as base a $\cQ^+(2d-2k+1, q)$, $\cQ(2d-2k+2, q)$, $\cH(2d-2k+1, q)$, respectively. Let $\Omega_d$ be the set of generators of $\cP_{d, e}$  contained in $\cK$. From Proposition~\ref{design}, Lemma~\ref{anti} and Equation \eqref{orthogonal}, we have 
$$
| \cR \cap \Omega_d | = \frac{m \prod_{i=1}^{k} (q^{d+e-i} + 1) \prod_{i=1}^{d-k+1} (q^{(d-k+1) + (e-1) -i} + 1)}{\prod_{i = 1}^{d} (q^{d+e-i} + 1)} = m (q^{e-1}+1).
$$
Hence, through $\sigma$ there pass $m(q^{e-1}+1)$ generators of $\cR'$. Let $\cR''$ be the set of generators of $\cR$ contained in a polar space $\cP_{d, e-1}''$ embedded in $\cP_{d, e}$, where $\cP_{d, e-1}''$ is $\cQ^+(2d-1, q)$, $\cQ(2d, q)$, $\cH(2d-1, q)$, respectively. We have just seen that through a $(k-2)$--space of $\cP_{d, e-1}''$ there pass $m(q^{e-1}+1)$ members of $\cR''$. Therefore $\cR''$ is an $m(q^{e-1}+1)$--regular system of $\cP_{d, e-1}''$ w.r.t. $(k-2)$--spaces.
\end{proof}

\begin{remark}
If $q$ is even, then the polar space $\cQ(2d, q)$, $d \ge 2$, has an $m$--regular system if and only if the polar space $\cW(2d-1, q)$ has an $m$--regular system. Indeed, let $N$ be the nucleus of $\cQ(2d, q)$ and let $\Pi$ be a hyperplane of $\PG(2d, q)$ not containing $N$. The projections from $N$ onto $\Pi$ of the totally singular $s$--dimensional subspaces of $\cQ(2d, q)$ are the totally isotropic $s$--dimensional subspaces of a polar space $\cW(2d-1, q)$ in $\Pi$. Hence an $m$--regular system of $\cQ(2d, q)$ is projected onto an $m$--regular system of $\cW(2d-1, q)$, and, conversely, any $m$--regular system of $\cW(2d-1, q)$ is the projection of an $m$--regular system of $\cQ(2d, q)$.
\end{remark}

\begin{cor}
\begin{itemize}
\item[1)] If $\cQ^-(2d + 1, q)$ has an $m$--regular system w.r.t. $(k-1)$--spaces, then $\cQ^+(2d+3,q)$ has a $2m(q+1)$--regular system w.r.t. $(k-1)$--spaces.
\item[2)] If $\cQ^-(2d + 1, q)$, $q$ even, has an $m$--regular system w.r.t. $(k-1)$--spaces, then $\cW(2d + 1, q)$ has an $m(q+1)$--regular system w.r.t. $(k-1)$--spaces.
\end{itemize}
\end{cor}
%

The generators of $\cQ^+(2d-1, q)$ are $(d-1)$--spaces and the set of all generators is partitioned into two subsets of the same size, called {\em Greek} and {\em Latin generators} and denoted by $\cM_1$ and $\cM_2$, respectively. It is straightforward to check that $\cM_i$ is a hemisystem of $\cQ^+(2d-1, q)$ w.r.t. $(k-1)$--spaces, $1 \le k \le d-1$. More interestingly it is possible to get regular systems by a switching procedure as described in the following proposition.

\begin{prop}
Let $\sigma$ be an $(s-1)$--space of $\cQ^+(2d-1, q)$, for a fixed $s$, with $1 \le s \le d-2$, and let $\cZ_i$ be the set of generators of $\cM_i$ containing $\sigma$, $i\in\{1,2\}$. Then $(\cM_i \setminus \cZ_i) \cup \cZ_j$, with $i\ne j$, is a hemisystem of $\cQ^+(2d-1, q)$ w.r.t. $(d-s-2)$--spaces.
\end{prop}
\begin{proof}
Without loss of generality, we can assume that $i=1$ and $j=2$. Since $|\cZ_1| = |\cZ_2|$, $\cZ_1 \subseteq \cM_1$ and $|\cZ_2 \cap \cM_1| = 0$, we have that $|(\cM_1 \setminus \cZ_1) \cup \cZ_2| = |\cM_1| = |\cM_2|$. Let $\sigma$ be an $(s-1)$--space of $\cQ^+(2d - 1, q)$ and let $\perp$ be the orthogonal polarity induced by the quadric. The projective subspace $\sigma^\perp$ meets the quadric $\cQ^+(2d-1, q)$ in a cone that has as vertex $\sigma$ and as base a $\cQ^+(2d-2s-1, q)$. Let $\rho$ be a $(d-s-2)$--space of $\cQ^+(2d-1, q)$. There are two cases to consider according as $\rho \not\subseteq \sigma^{\perp}$ or $\rho \subseteq \sigma^{\perp}$, respectively. In the former case, there are no generators of $\cZ_1 \cup \cZ_2$ containing $\rho$. Hence $(\cM_1 \setminus \cZ_1) \cup \cZ_2$ and $\cM_1$ have the same number of elements through $\rho$. In the latter case, if $\rho \cap \sigma$ is a $(t-1)$--space, then $\langle \rho, \sigma \rangle$ is a $(d - t - 2)$--space and $\langle \rho, \sigma \rangle^{\perp}$ meets the quadric $\cQ^+(2d-1, q)$ in a cone that has as vertex $\langle \rho, \sigma \rangle$ and as base a $\cQ^+(2t+1, q)$. Hence there is a bijection between the  generators of $\cZ_1 \cup \cZ_2$ passing through $\rho$ and the generators of $\cQ^+(2t+1, q)$; furthermore the members of $\cZ_1$ containing $\rho$ and those of $\cZ_2$ containing $\rho$ are equal in number. Note that the generators of $\cM_1$ passing through $\rho$ and $\sigma$ are elements of $\cZ_1$. Therefore we may conclude that the number of generators of $(\cM_1 \setminus \cZ_1) \cup \cZ_2$ through $\rho$ is the same as the number of generators of $\cM_1$ through $\rho$.
\end{proof}

In \cite{BLL} the authors proved that $\cW(5, q)$, (and hence $\cH(5, q^2)$) $q \in\{3, 5\}$, has not hemisystems w.r.t. lines. Moreover they exhibited some examples of hemisystems w.r.t. lines of $\cQ(6, q)$, $q \in\{3, 5\}$, \cite[Table 1]{BLL}. Both results were found with the aid of a computer. As already mentioned in the introduction, if $k \ge 2$ and the polar space is not a hyperbolic quadric, up to date, these are the only known sporadic examples of regular systems w.r.t. $(k-1)$--spaces. 

In the rest of the section we recall the state of the art about regular systems of polar spaces of rank two and then we apply Theorem \ref{chain} to obtain regular systems in polar spaces of higher rank. The point line dual of $\cW(3,q)$ is $\cQ(4,q)$, the parabolic quadric of $\PG(4,q)$. Therefore $m$--regular systems of $\cQ(4,q)$ and $m$--ovoids of $\cW(3,q)$ are equivalent objects and $m$--regular systems of $\cW(3,q)$ and $m$--ovoids of $\cQ(4,q)$ are equivalent objects. In odd characteristic, currently, $m$--ovoids of $\cQ(4,q)$ are rare. Indeed, infinite families of $m$--ovoids are known to exist only when $m=1$ \cite{PW}, $m = (q+1)/2$ \cite{CP} and $m = (q-1)/2$, \cite{FMX, FT}. See also \cite{BLP, Pa} for few more sporadic examples. As regard as $m$--ovoids of $\cW(3,q)$, then, if $q$ is odd, they exist for all even $m$, see \cite{BLP}.
The generalized quadrangle $\cW(3,q)$ is self--dual if and only if $q$ is even. Hence, when $q$ is even $m$--ovoids and $m$--regular systems of $\cW(3,q)$ (or $\cQ(4,q)$) are equivalent objects. Assume that $q$ is even. In \cite{CCEM}, the authors showed that $m$--ovoids (and hence $m$--regular systems) of $\cW(3,q)$ (and hence of $\cQ(4,q)$) exist for all integers $m$, $1 \le m \le q$.

The point line dual of $\cH(3,q^2)$ is $\cQ^-(5,q)$, the elliptic quadric of $\PG(5,q)$. Therefore $m$--regular systems of $\cQ^-(5,q)$ and $m$--ovoids of $\cH(3,q^2)$ are equivalent objects and $m$--regular systems of $\cH(3,q^2)$ and $m$--ovoids of $\cQ^-(5,q)$ are equivalent objects. It is a well known result that regular systems of $\cH(3,q^2)$ are hemisystems and several examples are known \cite{BGR, BLMX, CPa, CP, KNS}. On the other hand, $m$--ovoids of $\cH(3,q^2)$ exist for all possible $m$. Indeed, $\cH(3,q^2)$ admits a {\em fan}, i.e., a partition of the pointset into ovoids. We briefly mention the construction of a fan of $\cH(3,q^2)$ due to Brouwer and Wilbrink \cite{BW}: let $\perp$ be the unitary polarity of $\PG(3,q^2)$ induced by $\cH(3,q^2)$, let $P$ be a point of $\PG(3,q^2) \setminus \cH(3,q^2)$, let $X$ be a point of $P^\perp \cap \cH(3,q^2)$ and let $t$ be the unique tangent line at $X$ contained in $P^\perp$. Put $\cO_X = P^\perp \cap \cH(3,q^2)$, and $\cO_Y = ((Y^\perp \cap \cH(3,q^2)) \setminus P^\perp) \cup (PY \cap \cH(3,q^2))$ for $Y \in t \setminus \{ X \}$. Then each $\cO_Z$, is an ovoid, and $\{ \cO_Z \;\; | \;\; Z \in t \}$ is a fan. Finally, up to date, no infinite family of $m$--regular systems of $\cH(4,q^2)$ is known, although several examples have been found when $q \in \{2,3,4,5\}$ \cite{BDS}.

Applying Theorem \ref{chain}, we obtain the following results.

\begin{cor}
If $q$ is even, then $\cQ^+(5,q)$ has a $2m$--regular system w.r.t. points for all $m$, $1 \le m \le q+1$. If $q$ is odd, then $\cQ^+(5,q)$ has a $2m$--regular system w.r.t. points for all even $m$, $1 \le m \le q+1$.
\end{cor}

\begin{cor}
$\cQ(6,q)$ has an $m(q+1)$--regular system w.r.t. points for all $m$, $1 \le m \le q^2+1$.
\end{cor}

\begin{cor}
$\cQ^+(7,q)$ has a $2m(q+1)$--regular system w.r.t. points for all $m$, $1 \le m \le q^2+1$.
\end{cor}

\begin{cor}
$\cH(5,q^2)$ has an $m(q+1)$--regular system w.r.t. points for $q=2$ and $3 \le m \le 6$, $q=3$ and $3 \le m \le 25$, $q=4$ and $5 \le m \le 60$, $q=5$ and $m=30$.
\end{cor}

\section{Hemisystems of elliptic quadrics}\label{sec_3}

In \cite[Section 6]{BLP}, the authors construct hemisystems of $\cQ(4, q)$, $q$ odd, by partitioning the generators of $\cQ(4, q)$ into generators of hyperbolic quadrics $\cQ(3, q)$ embedded in $\cQ(4, q)$. In this section we investigate and generalize their idea, providing a construction of hemisystems of $\cQ^-(5, q)$, $q$ odd. Let $\perp$ be the polarity of $\PG(2n+1, q)$ associated with $\cQ^-(2n+1, q)$.

\begin{theorem}\label{th1}
Let $\P$ be a partition of the generators of the elliptic quadric $\cQ^-(2n+1, q)$, $n \ge 2$, into generators of hyperbolic quadrics $\cQ^+(2n-1, q)$ embedded in $\cQ^-(2n+1, q)$. Then $q$ is odd and $2^\frac{(q^n+1)(q^{n+1}+1)}{2(q+1)}$ hemisystems w.r.t. points of $\cQ^-(2n+1,q)$ arise, by taking one family from each of the Latin and Greek generator pairs in $\P$, and then considering the union of these generators.
\end{theorem}
\begin{proof}
Since the generators of $\cQ^-(2n +1, q)$ are $\prod_{i = 2}^{n+1}(q^i+1)$ and the generators of $\cQ^+(2n-1, q)$ are $\prod_{i = 0}^{n-1}(q^i+1)$, we have that the number of members in $\P$ is $\frac{(q^n+1)(q^{n+1}+1)}{2(q+1)}$ and hence $q$ has to be odd. Similarly through a point $P$ of $\cQ^-(2n +1, q)$ there pass $\prod_{i = 2}^{n}(q^i +1)$ generators of $\cQ^-(2n+1, q)$ and $\prod_{i = 0}^{n-2}(q^i+1)$ generators of $\cQ^+(2n-1, q)$, therefore every point of $\cQ^-(2n+1, q)$ is contained in a constant number of elements of $\P$. On the other hand, since the number of generators of a Latin family through $P$ equals the number of generators of a Greek family through $P$, by selecting one family from each Latin and Greek pair in $\P$ amounts to choose exactly half of the generators through $P$. Finally note that the number of hemisystems obtained by selecting one family from each of the elements of $\P$ equals $2^\frac{(q^n+1)(q^{n+1}+1)}{2(q+1)}$.
\end{proof}

\begin{prop}
Let $\cL$ be a set of $\frac{(q^n+1)(q^{n+1}+1)}{2(q+1)}$ lines external to $\cQ^-(2n+1, q)$ such that
$$
|\langle r, r' \rangle \cap \cQ^-(2n+1, q)| \ne
\begin{cases}
 1 & \mbox{ if } \;\; |r \cap r'| = 1, \\
 q+1 & \mbox{ if } \;\; |r \cap r'| = 0 ,
\end{cases}
$$
for each $r, r' \in \cL$, $r \ne r'$. Then there exists a partition of the generators of $\cQ^-(2n+1, q)$ into generators of a $\cQ^+(2n-1, q)$.
\end{prop}
\begin{proof}
The generators of the hyperbolic quadrics $r^\perp \cap \cQ^-(2n+1, q)$, with $r \in \cL$, form the partition.
\end{proof}

Below we show that there exists such a partition of the generators of $\cQ^-(5, q)$.

\bigskip
Let $\Pi$ be a solid of $\PG(5, q)$ meeting $\cQ^-(5, q)$ in a hyperbolic quadric $\cQ^+(3, q)$ and let $\ell = \Pi^\perp$. Let $\cX$ be the set of lines of $\Pi$ that are external to $\cQ^-(5, q)$. Let $\ell_1$ be a line of $\cX$ and let $\ell_2 = \ell_1^\perp \cap \Pi$. Denote by $\Pi_i = \ell_i^\perp$, $i = 1,2$. Since $\ell_i \in \cX$, $\Pi_i \cap \cQ^-(5, q)$ is a three--dimensional hyperbolic quadric. Let $\cX_1$ be the set of lines of $\Pi_1$ external to $\cQ^-(5, q)$ and intersecting $\ell$ in at least a point; let $\cX_2$ be the set of lines of $\Pi_2$ external to $\cQ^-(5, q)$ and meeting both $\ell$ and $\ell_1$ in exactly one point. Note that in $\PG(3, q)$ there pass $q(q-1)/2$ lines external to a $\cQ^+(3, q)$ through a point off the quadric. Thus $|\cX| = q^2(q-1)^2/2$, $|\cX_1| = (q-2)(q+1)^2/2+1$ and $|\cX_2| = (q+1)^2/2$ and the set $\cX \cup \cX_1 \cup \cX_2$ consists of $(q^2-q+1)(q^2+1)/2$ lines external to $\cQ^-(5, q)$.

\begin{remark}
Let $P$ be a point of $\ell_1 \cup \ell_2$ and let us denote by $\pi_P$ the plane spanned by $\ell$ and $P$. Note that $\pi_P \cap \cQ^-(5, q)$ is a conic of which $P$ is an internal point. Hence the lines intersecting both $\ell$ and $\ell_i$, in exactly one point are either external or secant and they are equals in number, for each $i \in\{1,2\}$.
\end{remark}

\begin{theorem}
For each $r, r' \in \cX \cup \cX_1 \cup \cX_2$, $r \ne r'$, we have
$$
|\langle r, r' \rangle \cap \cQ^-(5, q)| \ne
\begin{cases}
 1 & \mbox{ if } \;\; |r \cap r'| = 1, \\
 q+1 & \mbox{ if } \;\; |r \cap r'| = 0 ,
\end{cases}
$$

\end{theorem}
\begin{proof}
Let $r, r'$ be two distinct lines both lying in $\cX$ (or in $\cX_1$, or in $\cX_2$). Then $|\langle r, r' \rangle \cap \cQ^-(5, q)|$ equals $(q+1)^2$ or $q+1$, according as $|r \cap r'| = 0$ or $|r \cap r'| = 1$, respectively.

Assume that $r \in \cX$ and $r' \in \cX_2$. Let $P$ be the point $r' \cap \ell_1$. We have two possibilities: either $P \in r$ or $P \notin r$. In the former case since $|r \cap r'| = 1$, $\langle r', r \rangle$ is a plane contained in the three--space $\langle \ell, r \rangle$ and $\langle \ell, r \rangle^\perp = r^\perp \cap \ell^\perp = r^\perp \cap \Pi$ is a line of $\Pi$ external to $\cQ^-(5, q)$. Hence $|\langle r, r' \rangle \cap \cQ^-(5, q)| = q+1$. In the latter case, let $\Gamma$ be the hyperplane $\langle r, r', \ell \rangle = \langle r, P, \ell \rangle$. Then $\Gamma^\perp = \langle r, P \rangle^\perp \cap \Pi$ is a point of $\Pi$ off the quadric $\cQ^-(5, q)$ since $\langle r, P \rangle \cap \cQ^-(5, q)$ is a conic. Therefore $\Gamma \cap \cQ^-(5, q)$ is a parabolic quadric $\cQ(4, q)$ and $|\langle r, r' \rangle \cap \cQ^-(5, q)| \in \{q^2+q+1, (q+1)^2, q^2+1\}$.

Assume that $r \in \cX$ and $r' \in \cX_1$. If $r' = \ell$, then $\langle r, r' \rangle^\perp \in \cX$ and hence $|\langle r, r' \rangle \cap \cQ^-(5, q)| = (q+1)^2$. If $r'$ meets $\ell_2$ in a point, repeating the same argument used before we are done. Let $r'$ be a line meeting $\ell$ in a point and disjoint from $\ell_2$. Then the plane $\langle r', \ell \rangle$ intersects $\Pi$ in a point, say $T$. Let $\Gamma$ be the projective space $\langle r, r', \ell \rangle$. If $T \in r$, then $\Gamma = \langle r, r' \rangle$ is a solid, $\Gamma^\perp \in \cX$ and $|\langle r, r' \rangle \cap \cQ^-(5, q)| = (q+1)^2$. If $T \notin r$, then $\Gamma$ is the hyperplane $\langle r, T, \ell \rangle$ and as before $\Gamma^\perp = \langle r, T \rangle^\perp \cap \Pi$ is a point of $\Pi$ off the quadric $\cQ^-(5, q)$ since $\langle r, T \rangle \cap \cQ^-(5, q)$ is a conic. Therefore $\Gamma \cap \cQ^-(5, q)$ is a parabolic quadric $\cQ(4, q)$ and $|\langle r, r' \rangle \cap \cQ^-(5, q)| \in \{q^2+q+1, (q+1)^2, q^2+1\}$.

Assume that $r \in \cX_1$ and $r' \in \cX_2$. Let $|r \cap r'| = 1$. If $r = \ell$, then $|\langle r, r' \rangle \cap \cQ^-(5, q)| = q+1$. If $r \ne \ell$, then the plane $\langle r, r' \rangle$ is contained in the solid $\langle \ell_1, r \rangle$, where $|\langle \ell_1, r \rangle \cap \cQ^-(5, q)| = (q+1)^2$, since $\langle \ell_1, r \rangle^\perp = r^\perp \cap \Pi_1$ is an external line of $\Pi_1$. Therefore, we have again that $|\langle r, r' \rangle \cap \cQ^-(5, q)| = q+1$. Let $|r \cap r'| = 0$ and let $\Gamma$ be the projective space spanned by $\langle r, r', \ell \rangle$. Since $r$ is not contained in $\Pi_2$, then $\Gamma$ is a solid and $\ell_1 \not\subseteq \Gamma$. Hence $\langle \Gamma, \ell_1 \rangle$ is a hyperplane and $\langle \Gamma, \ell_1 \rangle^\perp \in \ell_2$. Therefore $\langle \Gamma, \ell_1 \rangle \cap \cQ^-(5, q)$ is a parabolic quadric $\cQ(4, q)$ and $|\langle r, r' \rangle \cap \cQ^-(5, q)| \in \{q^2+q+1, (q+1)^2, q^2+1\}$.
\end{proof}

\begin{prob}
Determine whether or not the generators of an elliptic quadric $\cQ^-(2n+1, q)$, $q$ odd, $n \ge 3$, can be partitioned into generators of hyperbolic quadrics $\cQ^+(2n-1, q)$ embedded in $\cQ^-(2n+1, q)$.
\end{prob}

\section{Regular systems arising from $k$--systems of polar spaces}\label{sec_4}

In this section we explore a class of regular systems w.r.t. points of a polar space $\cP$ that can be obtained by means of a $k$--system of $\cP$. The notion of $k$--system was introduced by Shult and Thas in \cite{ST}.

\begin{defin}
A $k$--system of a polar space $\cP_{d, e}$, $1 \le k \le d - 2$, is a set $\Pi_1, \dots, \Pi_{q^{d+e-1} + 1}$ of $k$--spaces of $\cP_{d, e}$, such that a generator of $\cP_{d, e}$ containing $\Pi_i$, is disjoint from $\cup_{j = 1, j \ne i}^{q^{d + e - 1} + 1} \Pi_j$.
\end{defin}

Let $\cS$ be a $k$--system of $\cP_{d, e}$ and let $\cG$ the set of generators of $\cP_{d, e}$ containing one element of $\cS$. It is not difficult to see that $\cG$ is a regular system of $\cP_{d, e}$ w.r.t. points.

\begin{lemma}\label{k-system}
The set $\cG$ is a $|\cM_{\cP_{d - k - 1, e}}|$--regular system of $\cP_{d, e}$ w.r.t. points.
\end{lemma}
\begin{proof}
Let $P$ be a point of $\cP_{d, e}$. If $P \in \Pi_i$, for some $i$, then there are $|\cM_{\cP_{d - k - 1, e}}|$ members of $\cG$ containing $P$. If $P$ is not contained in an element of $\cS$, then from \cite[Theorem 5]{ST}, there are $q^{d - k - 2 + e} + 1$ $(k+1)$--spaces of $\cP_{d, e}$ containing $P$ and a member of $\cS$. Hence the point $P$ belongs to $|\cM_{d - k - 2, e}| (q^{d - k  - 2 + e} + 1) = |\cM_{\cP_{d - k - 1, e}}|$ generators of $\cG$.
\end{proof}

\subsection{Regular systems of $\cQ^+(5,q)$, $q$ odd}

Here we turn our attention to the hyperbolic quadric of $\PG(5,q)$, $q$ odd. The hyperbolic quadric ${\cal Q}^+(5,q)$ of $\PG(5,q)$ is also known as the {\em Klein quadric}. There exists a correspondence, which we will refer to as {\em Klein correspondence}, between the lines of $\PG(3,q)$ and the points of $\cQ^+(5,q)$. In particular the $(q^2+1)(q^2+q+1)$ points of $\cQ^+(5,q)$ correspond one to one to the lines of $\PG(3,q)$. The quadric $\cQ^+(5,q)$ contains two families of planes, each of size  $q^3+q^2+q+1$ called {\em Latin planes} and {\em Greek planes}. A Latin plane and a Greek plane correspond to lines through a point and to lines contained in a plane of $\PG(3,q)$, respectively. Any two Greek planes and any two Latin planes share exactly one point. A Greek plane and a Latin plane are either disjoint or meet in a line \cite[Section 15.4]{H}.

A $1$--system of $\cQ^+(5, q)$ is a set $\cS$ of $q^2+1$ lines of $\cQ^+(5, q)$ such that no plane of $\cQ^+(5, q)$ through a line of $\cS$ has a point in common with the remaining lines of $\cS$. From \cite[Theorem 15]{ST}, the quadric $\cQ^+(5,q)$, $q$ odd, has a unique $1$--system $\cS$ and the points covered by the $q^2+1$ lines of $\cS$ correspond, under the Klein correspondence, to the lines that are tangent to an elliptic quadric of $\PG(3, q)$. In \cite{E}, Ebert provided a partition $\cF$ of $\PG(3,q)$ into $q+1$ disjoint three--dimensional elliptic quadrics, say $\cE_0, \dots, \cE_q$. This result is achieved by considering a subgroup of order $q^2+1$ of a Singer cyclic group of $\PGL(4,q)$. In particular, if $q$ is odd, we have the following result.

\begin{lemma} \label{ebert} \cite[Corollary 1, Theorem 5]{E}
The lines of $\PG(3,q)$, $q$ odd, are partitioned as follows:
\begin{itemize}
\item[i)] $(q^2+1)^2/2$ lines tangent to no elliptic quadrics of $\cF$ and hence secant to $(q+1)/2$ of them;
\item[ii)] $(q+1)^2(q^2+1)/2$ lines tangent to two elliptic quadrics of $\cF$ and hence secant to $(q-1)/2$ of them.

Also,
\item[iii)] no plane of $\PG(3, q)$ is tangent to two distinct elliptic quadrics of $\cF$;
\item[iv)] there exists a line of $\PG(3,q)$ that is tangent to $\cE_i$ and $\cE_j$, with $0 \le i, j \le q$, $i \ne j$, if and only if $i+j$ is odd.
\end{itemize}
\end{lemma}

Let $\cL_i$ be the set consisting of the $(q+1)(q^2+1)$ lines that are tangent to the elliptic quadric $\cE_i$ and let $\cS_i$ be the corresponding $1$--system of $\cQ^+(5,q)$, obtained by applying the Klein correspondence to the lines of $\cL_i$, $0 \le i \le q$. By using Lemma \ref{ebert}, we are able to prove the following result.

\begin{theorem}\label{klein}
There is no plane of $\cQ^+(5,q)$, $q$ odd, containing two lines of $\cup_{i = 0}^q \cS_i$.
\end{theorem}
\begin{proof}
Assume by contradiction that a generator of $\cQ^+(5, q)$, say $\pi$, contains two lines, say $ \ell_1 \in \cS_i$ and $\ell_2 \in \cS_j$. Note that necessarily $i \ne j$. Under the inverse of the Klein correspondence, the points of $\ell_1$ are mapped to the lines that are tangent to $\cE_i$ at a point, say $P_1$, and the points of $\ell_2$ are mapped to the lines that are tangent to $\cE_j$ at a point, say $P_2$. If $\pi$ is a Latin plane, then $P_1 = P_2$, contradicting the fact that $|\cE_i \cap \cE_j| = 0$. If $\pi$ is a Greek plane, then would exist a plane of $\PG(3, q)$ tangent to two distinct elliptic quadrics of $\cF$, contradicting {\em iii)} of Lemma \ref{properties}.
\end{proof}

From Lemma \ref{k-system} and {\em iv)} of Lemma \ref{properties}, we have the following result.

\begin{cor}
$\cQ^+(5,q)$, $q$ odd, has a $2m$--regular system w.r.t. points for all $m$, $1 \le m \le (q+1)$.
\end{cor}

\begin{remark}
From {\em iv)} of Lemma \ref{ebert}, it follows that the $(q+1)(q^2+1)/2$ lines of $\cup_{i = 0}^{(q-1)/2} \cS_{2i}$ are pairwise disjoint. Similarly for the lines of $\cup_{i = 0}^{(q-1)/2} \cS_{2i + 1}$. Moreover these two sets of lines cover the same points.
\end{remark}

\subsection{Regular systems of $\cQ(6, 3)$}

Let $\cQ(6, q)$ be a parabolic quadric of $\PG(6, q)$ and let $\PGO(7, q)$ be the group of projectivities of $\PG(6, q)$ leaving invariant $\cQ(6, q)$. A $1$--system of $\cQ(6, q)$ is a set $\cS$ of $q^3+1$ lines of $\cQ(6, q)$ such that no plane of $\cQ(6, q)$ through a line of $\cS$ has a point in common with the remaining lines of $\cS$. From Lemma \ref{k-system}, the planes of $\cQ(6, q)$ containing a line of $\cS$ form a $(q+1)$--regular system w.r.t. points. If $\cS$ is the union of $q^2-q+1$ reguli, let $\cS^o$ be the $1$--system obtained by taking the $q^2-q+1$ opposite reguli of $\cS$. In this case the previous regular system can be obtained twice. 

\begin{prop}
If $\cS$ is a $1$--system of $\cQ(6, q)$ obtained from the union of $q^2-q+1$ reguli, then the set of planes of $\cQ(6, q)$ containing a line of $\cS \cup \cS^o$ is a $2(q+1)$--regular system w.r.t. points of $\cQ(6, q)$.
\end{prop}
\begin{proof}
By construction $\cS^o$ is a $1$--system of $\cQ(6, q)$ and $|\cS \cap \cS^o| = 0$. Moreover a generator of $\cQ(6, q)$ containing a line of $\cS$ contains no line of $\cS^o$, otherwise there would be two lines of $\cS$ meeting in one point, a contradiction. Therefore from Lemma \ref{properties}, {\em iv)}, the set of generators of $\cQ(6, q)$ containing a line of $\cS \cup \cS^o$ is a $2(q+1)$--regular system w.r.t. points of $\cQ(6, q)$.
\end{proof}

Below we give a construction of a $1$--system of $\cQ(6, 3)$ that is the union of $7$ reguli and it is not contained in an elliptic hyperplane of $\cQ(6, 3)$. The group $\PGO(7, 3)$ has two orbits on points of $\PG(6, 3)$ not on $\cQ(6, 3)$, i.e., the set $I$ of internal points and the set $E$ of external points. 
Let $\perp$ be the orthogonal polarity of $\PG(6, 3)$ defining $\cQ(6, 3)$. If $P \in \PG(6, 3) \setminus \cQ(6, 3)$, we have that $P^\perp \cap \cQ(6, 3)$ is an elliptic quadric $\cQ^-(5, 3)$ or a hyperbolic quadric $\cQ^+(5, 3)$, according as $P$ belongs to $I$ or $E$, respectively.

\begin{construction}\label{sys_para}
Let $\cX = \{P_1, \dots, P_7\} \subset I$ such that $\cX$ is a self--polar simplex of $\cQ(6, 3)$, i.e., $P_i^\perp= \langle P_j: \;\; 1 \le j \le 7,  i \ne j\rangle$, with $1 \le i \le 7$. A line containing two points of $\cX$ is external to $\cQ(6, 3)$, a plane $\sigma$ spanned by three points of $\cX$ meets $\cQ(6, 3)$ in a non--degenerate conic and $\sigma^\perp \cap \cQ(6, 3)$ is a hyperbolic quadric $\cQ^+(3, 3)$. Let $\pi = \langle P_1, P_2, P_3 \rangle$ and consider the following $7$ lines: $r_1 = \langle P_1, P_2 \rangle$, $r_2 = \langle P_2, P_3 \rangle$, $r_3 = \langle P_3, P_1 \rangle$, $\ell_1 = \langle P_4, P_5 \rangle$, $\ell_1' = \langle P_6, P_7 \rangle$, $\ell_2 = \langle P_4, P_7 \rangle$, $\ell_2' = \langle P_5, P_6 \rangle$, $\ell_3 = \langle P_4, P_6 \rangle$, $\ell_3' = \langle P_5, P_7 \rangle$. Let $\varphi$ be a permutation of $\{1, 2, 3\}$. Let $\cR_i$ be one of the two reguli of the hyperbolic quadric $\langle r_i, \ell_{\varphi(i)} \rangle \cap \cQ(6, 3)$, $\cR_i'$ one of the two reguli of $\langle r_i, \ell_{\varphi(i)}' \rangle \cap \cQ(6, 3)$ and $\cR$ one of the two reguli of $\pi^\perp \cap \cQ(6, 3)$. Set $\cS = \left(\bigcup_{i = 1}^{3} \cR_i \cup \cR_i'\right) \cup \cR$.
\end{construction}

\begin{prop}
The set $\cS$ is a $1$--system of $\cQ(6, 3)$.
\end{prop}
\begin{proof}
It is enough to show that the $5$--space generated by two $3$--spaces $T_1$, $T_2$, containing two of the seven reguli of $\cS$ meets $\cQ(6, 3)$ in an elliptic quadric $\cQ^-(5, 3)$. By construction, $T_{1}$ and $T_{2}$ meet in a line containing two points of $\cX$. Therefore, $\langle T_{1}, T_{2} \rangle$ is a $5$--space, containing six points of $\cX$ and hence $\langle T_{1}, T_{2} \rangle^{\perp}$ is the remaining point of $\cX$, that is an internal point of $\cQ(6, 3)$. It follows that $\langle T_{1},T_{2} \rangle \cap \cQ(6, 3)$ is an elliptic quadric $\cQ^{-}(5, 3)$. 
\end{proof}
Let $\cS^o = \left(\bigcup_{i = 1}^{3} \cR_i^o \cup \cR_i'^o\right) \cup \cR^o$, where $\cR_i^o$, $\cR_i'^o$ and $\cR^o$ are the opposite regulus of $\cR_i$, $\cR_i'$ and $\cR$, respectively. Then the set of generators of $\cQ(6, 3)$ containing a line of $\cS \cup \cS^o$ is an $8$--regular system of $\cQ(6, 3)$ w.r.t. points.

Some computations performed with Magma \cite{magma} show that, up to projectivities, there are two distinct $1$--systems of $\cQ(6, 3)$ arising from Construction \eqref{sys_para}. The former is stabilized by a subgroup of $\PGO(7, 3)$ of order $1344$ isomorphic to $2^3 \cdot \PSL(3, 2)$ and it acts transitively on members of $\cS$. Hence in this case $\cS$ is the $1$--system described in \cite[Section 3.4]{BP}, \cite{DHV}. The latter is stabilized by a subgroup of $\PGO(7, 3)$ of order $192$ that is not transitive on lines of $\cS$. To the best of our knowledge it is new.

\begin{prob}
Construct a $1$--system of $\cQ(6, q)$ that is the union of $q^2-q+1$ reguli and it is not contained in an elliptic hyperplane of $\cQ(6, q)$.
\end{prob}

\section{Regular systems arising from field reduction}\label{sec_5}

Let $V$ be a vector space of dimension $N$ over the field $\GF(q^r)$. Since $\GF(q) \subset \GF(q^r)$, we have that $V$ can be seen as a vector space of dimension $rN$ over the field $\GF(q)$. This means that there exists a bijection $\phi$ between the vectors of $V$ and $V'$ which maps $k$--dimensional  subspaces of $V$ to certain $kr$--dimensional  subspaces of $V'$. In this setting if $U$ denotes the set of singular vectors (totally singular subspaces) of a non--degenerate quadratic form $Q$ on $V$ or isotropic vectors (totally isotropic subspaces) of a non--degenerate reflexive sesquilinear form $\beta$ on $V$, then $\phi(U) = \{\phi(u) \;\; | \;\; u \in U\}$ are vectors (subspaces) of $V'$ which could be either singular (totally singular) with respect to a quadratic form of $V'$ or isotropic (totally isotropic) with respect to a reflexive sesquilinear form on $V'$. Moreover, if $f$ is a semilinear map of $V$, then there exists a unique semilinear map $f'$ of $V'$ such that $\phi f = f' \phi$. Analagously, a semi--similarity of $Q$ or $\beta$ gives rise to a semi--similarity of the corresponding form of $V'$. Hence there is a natural embedding of a subgroup of $\GaL(V)$ in $\GaL(V')$ and of a subgroup of semi--similarities of $Q$ or $\beta$ in the group of semi--similarities of the corresponding form of $V'$. The subgroups of $\GaL(V')$ arising in this way belong to the third--class described by Aschbacher \cite{A, K}.

From a projective point of view $\phi$ maps $(k-1)$--spaces of $\PG(V)$ to certain $(kr-1)$--spaces of $\PG(V')$ and if $\cP$ denotes the non--degenerate polar space of $\PG(V)$ associated with $Q$ or $\beta$, then $\phi$ sends $(k-1)$--spaces of $\cP$ to certain $(kr-1)$--spaces of a polar space $\cP'$ of $\PG(V')$. The map $\phi$ is usually called {\em field reduction map} and it has been widely studied by many authors \cite{G, LV, V}. Indeed, it is the tool which allows to construct Desarguesian $m$--spreads of $\PG(V')$ \cite{Segre} or $\cP'$ \cite{D1,D2,D3,D,D4,D5} and ``classical'' $m$-systems of $\cP'$, see \cite{ST, ST1}.

The content of this section is based on the following observation. If $G$ is a group of collineations of $\cP'$ acting transitively on points of $\cP'$ and $\cO$ is an orbit on the generators of $\cP'$ under the action of $G$, then then through each point of $\cP'$ there will be a constant number of elements of $\cO$, i.e., $\cO$ is a regular system w.r.t. points of $\cP'$. See also \cite{BE}. On the other hand, if $\cD$ is a Desarguesian $m$--spread of $\cP'$, by Witt's theorem it follows that there exists a group of projectivities of $\cP'$, say $G_{\cD}$, preserving $\cD$ and acting transitively on the points of $\cP'$. Therefore every orbit of $G_{\cD}$ on the generators of $\cP'$ is a regular system w.r.t. points of $\cP'$. In particular in order to determine regular systems of $\cP'$ w.r.t. points arising in this way, a comprehensive study of the action of $G_{\cD}$ on the generators of $\cP'$ is needed. 

Here we focus on the case when $\cP$ is a non--degenerate Hermitian variety $\cH(N-1, q^2)$ of $\PG(N-1, q^2)$ and $\cP'$ is a non--degenerate quadric $\cQ$ of $\PG(2N-1, q)$ with a Desarguesian line--spread $\cL_1$. In the projective orthogonal group of $\cQ$ there are two subgroups isomorphic to $\SU(N, q^2)/\langle -I \rangle$ and $\U(N, q^2)/\langle -I \rangle$ preserving $\cL_1$. The action of these two subgroups on generators of $\cQ$ is determined.

Some preliminary results of the following sections overlap the content of other papers (see for instance \cite{D}); however we prefer to include here a proof for the convenience of the reader.

\subsection{Subgeometries embedded in $\cH(2n-1, q^2)$}

Let $\cH(2n-1, q^2)$ be the non--degenerate Hermitian variety of $\PG(2n-1, q^2)$ associated with the Hermitian form on $V(2n, q^2)$ given by
$$
\begin{pmatrix} X_1, \dots, X_{2n} \end{pmatrix} J \begin{pmatrix} Y_1^q \dots Y_{2n}^q \end{pmatrix}^t,
$$
where $J^q = J^t$ and let $\perp$ be the corresponding Hermitian polarity of $\PG(2n-1, q^2)$. Consider the following groups of unitary isometries
\begin{align*}
\U(2n, q^2) = \{M \in \GL(2n, q^2) \;\; | \;\; M^t J M^q = J\}, \\
\SU(2n, q^2) = \{M \in \SL(2n, q^2) \;\; | \;\; M^t J M^q = J\},
\end{align*}
and the groups of projectivities induced by them, namely $\PU(2n, q^2)$ and $\PSU(2n, q^2)$, respectively. Let $\Sigma$ be a Baer subgeometry of $\PG(2n-1, q^2)$ and let $\tau$ be the Baer involution of $\PG(2n-1, q^2)$ fixing pointwise $\Sigma$. We say that $\cH(2n-1, q^2)$ {\em induces a polarity on $\Sigma$} if $\perp$, when restricted on $\Sigma$, defines a polarity of $\Sigma$, or equivalently if $P^\perp \cap \Sigma$ is a hyperplane of $\Sigma$, for every point $P \in \Sigma$.

\begin{lemma}
$\cH(2n-1, q^2)$ induces a polarity on $\Sigma$ if and only if $\tau$ fixes $\cH(2n-1, q^2)$.
\end{lemma}
\begin{proof}
Suppose first that $\cH(2n-1,q^2)$ induces a polarity on $\Sigma$. Since $\PGL(2n-1, q^2)$ is transitive on the Baer subgeometries of $\PG(2n-1, q^2)$, we may assume without loss of generality that $\Sigma$ is the canonical subgeometry of $\PG(2n-1,q^2)$. Hence $\cH(2n-1, q^2)$ is associated with the Hermitian form $\begin{pmatrix} X_1, \dots, X_{2n} \end{pmatrix} J \begin{pmatrix} Y_1^q \dots Y_{2n}^q \end{pmatrix}^t$, whereas the polarity of $\Sigma$ is defined by the form $\begin{pmatrix} X_1, \dots, X_{2n} \end{pmatrix} J \begin{pmatrix} Y_1 \dots Y_{2n} \end{pmatrix}^t$. Thus $J^q = J^t = \pm J$. Since $I^t J I = \pm J^q$, $\tau$ fixes $\cH(2n-1,q^2)$.

On the other hand, let $\tau$ be the composition of the projectivity represented by a matrix $M \in \GL(2n,q^2)$ with the non--linear involution sending the point $P$ to the point $P^q$. Since $\tau$ fixes $\cH(2n-1,q^2)$, we have that $M^t J M^q = \lambda J^q$, for some $\lambda \in \GF(q) \setminus \{0\}$. Replacing $M$ with $a M$, where $a^{q+1} = \lambda$, we have that $M^t J M^q = J^q$. We claim that $J M^q = M^{-t} J^q$ defines a polarity of $\Sigma$. Indeed, since $\tau$ is an involution, then $M M^q = \sigma I$, for some $\sigma \in \GF(q^2) \setminus \{0\}$ and
$$
\sigma J (P^q)^t = J (M (P^q)^t)^q = J M^q P^t = M^{-t} J^q  {P}^t = M^{-t} (J (P^q)^t)^q.
$$
This means that $P^\perp = (P^\tau)^\perp = (P^\perp)^\tau$. In particular $P^\perp$ is fixed by $\tau$ and hence $P^\perp$ is a hyperplane of $\Sigma$, as required.
\end{proof}

Note that if $\cH(2n-1, q^2)$ induces a polarity on $\Sigma$, then $\perp|_{\Sigma}$ is either symplectic or pseudo--symplectic if $q$ is even and is either symplectic or orthogonal if $q$ is odd. Indeed, with the same notation used above, if $M M^q = \sigma I$ and $M^t J M^q = J^q$, then $M M^q = M J^{-1} M^{-t} J^q = \sigma I$ and hence $J M^{-1} J^{-t} M^t = \sigma I$. Thus $M J^{-1} = \sigma (J^{q})^{-1} M^t$ and $J^{-t} M^t = \sigma M J^{-1}$. Therefore $M J^{-1} = \sigma (J^{q})^{-1} M^t = \sigma J^{-t} M^t = \sigma (\sigma M J^{-1}) = \sigma^2 M J^{-1}$ and $\sigma^2 = \pm 1$. The polarity $\perp|_{\Sigma}$ of $\Sigma$, when extended in $\PG(2n-1, q^2)$ gives rise to a polarity $\perp'$ of $\PG(2n-1, q^2)$ associated with the non--degenerate bilinear form given by $\begin{pmatrix} X_1, \dots, X_{2n} \end{pmatrix} J M^q \begin{pmatrix} Y_1 \dots Y_{2n} \end{pmatrix}^t$ such that
\begin{itemize}
\item[{\em i)}] $\perp'$ is of the same type as $\perp|_{\Sigma}$,
\item[{\em ii)}] $\perp \perp' = \perp' \perp = \tau$.
\end{itemize}
The polarity $\perp'$ is said to be {\em permutable with $\perp$}, see also \cite{Segre1}. We say that $\Sigma$ {\em is embedded in} $\cH(2n-1, q^2)$ if $P \in \cH(2n-1, q^2)$, for every point $P \in \Sigma$.

\begin{lemma}
$\Sigma$ is embedded in $\cH(2n-1, q^2)$ if and only if $\cH(2n-1, q^2)$ induces a symplectic polarity on $\Sigma$.
\end{lemma}
\begin{proof}
If $\Sigma$ is embedded in $\cH(2n-1, q^2)$ and $P$ is a point of $\cH(2n-1, q^2)$ with $P \notin \Sigma$, then there exists a unique line $\ell$ of $\PG(2n-1, q^2)$ such that $P \in \ell$ and $|\ell \cap \Sigma| = q+1$. Moreover $P^\tau \in \ell$ and $P^\tau \in \cH(2n-1, q^2)$, since $\ell$ is a line of $\cH(2n-1, q^2)$. Therefore $\tau$ fixes $\cH(2n-1, q^2)$, $\cH(2n-1, q^2)$ induces a polarity on $\Sigma$ and such a polarity has to be symplectic since $\Sigma \subset \cH(2n-1, q^2)$. Viceversa if $\cH(2n-1, q^2)$ induces a symplectic polarity on $\Sigma$, then $\Sigma$ is embedded in $\cH(2n-1, q^2)$.
\end{proof}

\begin{prop}
There are $q^{n^2-n} \prod_{i = 2}^{n} (q^{2i-1} + 1)$ Baer subgeometries embedded in $\cH(2n-1, q^2)$.
\end{prop}
\begin{proof}
First observe that if $\ell$ is a line containing $q+1$ points of $\cH(2n-1, q^2)$, then $\ell^\perp \cap \cH(2n-1, q^2) = \cH(2n-3, q^2)$ and through a Baer subgeometry $\Sigma$ isomorphic to $\PG(2n-3, q)$ embedded in $\cH(2n-3, q^2)$ there pass exactly $q+1$ Baer subgeometries isomorphic to $\PG(2n-1, q)$ and embedded in $\cH(2n-1, q^2)$. Indeed, if $\Sigma'$ is a Baer subgeometry embedded in $\cH(2n-1, q^2)$, with $\Sigma \subset \Sigma'$, then $\ell \cap \cH(2n-1, q^2) = \ell \cap \Sigma'$ and every line joining a point $P' = \ell \cap \Sigma'$ and a point $P \in \Sigma$ meets $\Sigma'$ in a Baer subline. On the other hand, on the line $P P'$ there are exactly $q+1$ Baer sublines containing both $P$ and $P'$ and if $s$ is one of such Bear sublines, then there is a unique Baer subgeometry $\Sigma'$ containing $\Sigma$ and $s$. Moreover since every line joining a point of $\ell \cap \Sigma'$ and a point of $\Sigma$ is a line of $\cH(2n-1, q^2)$ we have that these subgeometries are embedded in $\cH(2n-1, q^2)$.

Let $x_n$ be the number of Baer subgeometries embedded in $\cH(2n-1, q^2)$. Since the number of $\cH(2n-3, q^2)$ in $\cH(2n-1, q^2)$ equals $q^{4n-4}(q^{2n-1}+1)(q^{2n}-1)/(q+1)(q^2-1)$ and the number of $\cW(2n-3, q)$ in $\cW(2n-1, q)$ equals $q^{2n-2}(q^{2n}-1)/(q^2-1)$, a standard double counting argument on couples $(\Sigma, \Sigma')$, where $\Sigma$ is a Baer subgeometry isomorphic to $\PG(2n-3, q)$ embedded in some $\cH(2n-3, q^2) \subset \cH(2n-1, q^2)$, $\Sigma'$ is a Baer subgeometry isomorphic to $\PG(2n-1, q)$ and embedded in $\cH(2n-1, q^2)$, with $\Sigma \subset \Sigma'$, gives $x_{n} = q^{2n-2}(q^{2n-1}+1) x_{n-1}$. The result follows from the fact that $x_1 = 1$.
\end{proof}

\begin{prop}\label{stabilizer1}
The group $\PU(2n, q^2)$ acts transitively on Baer subgeometries embedded in $\cH(2n-1, q^2)$.
\end{prop}
\begin{proof}
Since $\PGL(2n-1, q^2)$ is transitive on non--degenerate Hermitian varieties of $\PG(2n-1, q^2)$, we may assume without loss of generality that $J$ is the matrix 

\begin{equation*}
\begin{pmatrix}
0 & \xi & \dots & 0 & 0 \\
\xi^q & 0 & \dots & 0 & 0 \\
\vdots & \vdots & \ddots & \vdots & \vdots \\
0 & 0 & \dots & 0 & \xi \\
0 & 0 & \dots & \xi^q & 0
\end{pmatrix},
\end{equation*}
where $0 \ne \xi \in \GF(q^2)$ is such that $\xi^q = -\xi$. Let $\Sigma$ be the canonical Baer subgeometry of $\PG(2n-1, q^2)$. Then $\Sigma$ is embedded in $\cH(2n-1, q^2)$ and a symplectic polarity is induced on $\Sigma$. We denote by $\tilde\Sigma$ the set of all non--zero vectors of $V(2n, q^2)$ representing the points of $\Sigma$. If $S \in \U(2n, q^2)$ and $S^q = S$, then $S$ fixes $\tilde\Sigma$ and hence $S\in\Sp(2n, q)$. The $\Sp(2n,q)$ is a subgroup of $\U(2n, q^2)$ stabilizing $\tilde\Sigma$. Let $N \in \U(2n, q^2)$ stabilizes $\tilde\Sigma$ and let $u_i$ be the vector having $1$ in the $i$--th position and $0$ elsewhere. Then $N u_i^t = \rho_i x_i^t$, $1 \le i \le 2n$, where $\rho_i \in \GF(q^2) \setminus \{0\}$ and $x_i^q = x_i$. Since there exists a matrix $S \in \Sp(2n, q)$ such that $S^{-1} u_i^t = x_i^t$, we have that $S N u_i^t = \rho_i u_i^t$. Moreover $S N$ stabilizes $\tilde\Sigma$. Therefore $\rho_i = \lambda_i \rho_1$, $2 \le i \le 2n$, where $\lambda_i \in \GF(q) \ne \{0\}$ and $S N = \rho_1 \diag(1, \lambda_2, \dots, \lambda_{2n})$. Straightforward calculations show that $(SN)^t J (SN)^q = J$ if and only if $\rho_1 = \lambda_2^{-1}$ and $\lambda_{2j} = \lambda_2 \lambda_{2j-1}^{-1}$, $2 \le j \le n$, whereas $\diag(1, \lambda, \lambda_3, \lambda_2 \lambda_3^{-1}, \dots, \lambda_{2n-1}, \lambda_2 \lambda_{2n-1}^{-1}) \in \Sp(2n, q^2)$ if and only if $\lambda_2 = 1$. It follows that
$$
Stab_{\U(2n, q^2)}(\tilde\Sigma) = \langle \Delta, \Sp(2n, q) \rangle,
$$
where $\Delta = \{\Delta_{\rho} = \diag(\rho, \rho^{-q}, \dots, \rho, \rho^{-q}) \;\; | \;\; \rho \in \GF(q^2) \setminus \{0\}\}$. Since $\Delta_{\rho} \in \Sp(2n, q)$ if and only if $\rho \in \GF(q) \setminus \{0\}$, we have that $|Stab_{\U(2n, q^2)}(\tilde\Sigma)| = (q+1) |\Sp(2n, q)| = q^{n^2} (q+1) \prod_{i = 1}^{n} (q^{2i}-1)$. The centre of $\langle \Delta, \Sp(2n, q) \rangle$ consists of $\{\Delta_{\rho} \;\; | \;\; \rho^{q+1} = 1\}$ . Hence $|Stab_{\PU(2n, q^2)}(\Sigma)| = q^{n^2} \prod_{i = 1}^{n} (q^{2i}-1)$. The result follows from the fact that $|\PU(2n, q^2)|/|Stab_{\PU(2n, q^2)}(\Sigma)|$ equals the number of Baer subgeometries embedded in $\cH(2n-1, q^2)$.
\end{proof}

\begin{prop}\label{stabilizer2}
The group $\PSU(2n, q^2)$ has $gcd(q+1, n)$ equally sized orbits on Baer subgeometries embedded in $\cH(2n-1, q^2)$.
\end{prop}
\begin{proof}
With the same notation used in Proposition \ref{stabilizer1}, since $\Sp(2n, q) \le \SU(2n, q^2)$ and $\Delta \cap \SU(2n, q^2) = \overline{\Delta} = \{\Delta_{\rho} \;\; | \;\; \rho^{n(q-1)} = 1\}$, we have that $Stab_{\SU(2n, q^2)}(\tilde\Sigma) = \langle \overline{\Delta}, \Sp(2n, q) \rangle$. Moreover $|Stab_{\SU(2n, q^2)}(\tilde\Sigma)| = gcd(q+1, n) q^{n^2} \prod_{i = 1}^{n} (q^{2i}-1)$, since $|\overline{\Delta}| = (q-1) gcd(q+1, n)$ and $|\overline{\Delta} \cap \Sp(2n, q)| = q-1$.

The centre of $Stab_{\SU(2n, q^2)}(\tilde\Sigma)$ consists of $\{\Delta_{\rho} \;\; | \;\; \rho^{n(q-1)} = \rho^{q+1} = 1\}$ and hence has size $gcd(q+1, 2n)$. It follows that $|Stab_{\PSU(2n, q^2)}(\Sigma)| = gcd(q+1, n) q^{n^2} \prod_{i = 1}^{n} (q^{2i}-1)/gcd(q+1, 2n)$ and $|\Sigma^{\PSU(2n, q^2)}| =  q^{n^2-n} \prod_{i = 2}^{n} (q^{2i-1} + 1)/gcd(q+1, n)$. If $\Sigma'$ is a subgeometry embedded in $\cH(2n-1, q^2)$, then $Stab_{\SU(2n, q^2)}(\tilde\Sigma')$ is conjugate to $Stab_{\SU(2n, q^2)}(\tilde\Sigma)$ in $\U(2n, q^2)$ and $Stab_{\PSU(2n, q^2)}(\Sigma')$ is conjugate to $Stab_{\PSU(2n, q^2)}(\Sigma)$ in $\PU(2n, q^2)$. Moreover $\SU(2n, q^2) \trianglelefteq \U(2n, q^2)$ and $\PSU(2n, q^2) \trianglelefteq \PU(2n, q^2)$. Hence every orbit of $\PSU(2n, q^2)$ on Baer subgeometries embedded in $\cH(2n-1, q^2)$ has the same size.
\end{proof}

\subsection{The Desarguesian line--spread of a quadric in odd dimension}

Let $\alpha \in \GF(q)$ such that $X^2 - X - \alpha$ is irreducible over $\GF(q)$. Let $w \in \GF(q^2)$ such that $w^2 = w + \alpha$. Then every element $x \in \GF(q^2)$ can be uniquely written as $x = y + w z$, where $y, z \in \GF(q)$. Moreover $w + w^q = 1$ and $w^{q+1} = - \alpha$. Consider the field reduction map:
$$
\phi: \x = (x_1, \dots, x_N) \in V(N, q^2) \longmapsto \y = (y_1, z_1, \dots, y_N, z_N) \in V(2N, q).
$$
Thus
$$
\phi \left( \langle (x_1, \dots, x_N) \rangle_{q} \right) = \langle (y_1, z_1, \dots, y_N, z_N) \rangle_{q}.
$$
Moreover
$$
\phi \left( \langle (x_1, \dots, x_N) \rangle_{q^2} \right) = \langle \y = (y_1, z_1, \dots, y_N, z_N), \tilde{\y} = (\alpha z_1, y_1 + z_1, \dots, \alpha z_N, y_N + z_{N}) \rangle_{q}
$$
and
$$
\cL = \left\{ \phi \left( \langle \x \rangle_{q^2} \right) \;\; | \;\; \0 \ne \x \in V(N, q^2) \right\}
$$
is a Desarguesian line--spread of $\PG(2N-1, q^2)$. Let $\xi = 2w - 1$. Hence $\xi^q = -\xi$ and $\xi^2 = 4 \alpha + 1$.

\subsubsection{The hyperbolic case}

Assume that $N = 2n$. Consider the following non--degenerate Hermitian form $H$ on $V(2n, q^2)$
$$
H(\x, \x') = \x J_n ({\x'}^q)^t = \begin{pmatrix} x_1, \dots, x_{2n} \end{pmatrix}
J_n
\begin{pmatrix} x_1'^q, \dots, x_{2n}'^q \end{pmatrix}^t ,
$$
where $J_n$ is given by
\begin{equation}\label{hermitian_matrix}
\begin{pmatrix}
0 & \xi & \dots & 0 & 0 \\
\xi^q & 0 & \dots & 0 & 0 \\
\vdots & \vdots & \ddots & \vdots & \vdots \\
0 & 0 & \dots & 0 & \xi \\
0 & 0 & \dots & \xi^q & 0
\end{pmatrix}.
\end{equation}
Denote by $J'_{n}$ the skew symmetric matrix $\xi^{-1} J_n$ and consider the following groups of unitary isometries or symplectic similarities
\begin{align*}
\U(2n, q^2) = \{M \in \GL(2n, q^2) \;\; | \;\; M^t J_n M^q = J_n\}, \\
\SU(2n, q^2) = \{M \in \SL(2n, q^2) \;\; | \;\; M^t J_n M^q = J_n\}, \\
\Sp(2n, q) = \{M \in \GL(2n, q) \;\; | \;\; M^t J_n' M = J_n'\}.
\end{align*}
We have that
\begin{align*}
 H(\x, \x) & = \xi (x_1 x_2^q - x_1^q x_2 + \ldots + x_{2n-1} x_{2n}^q - x_{2n-1}^q x_{2n}) \\
  & = - \xi^2 (y_1 z_2 - z_1 y_2 + \ldots + y_{2n-1} z_{2n} - z_{2n-1} y_{2n}) = Q(\y).
\end{align*}
Note that $Q$ is a non--degenerate quadratic form on $V(4n, q)$ of hyperbolic type. Let $\cH(2n-1, q^2)$ be the Hermitian variety of $\PG(2n-1, q^2)$ determined by $H$. Denote by $\cQ^+(4n-1, q)$ the hyperbolic quadric of $\PG(4n-1, q)$ defined by $Q$ and let $\PGO^+(4n, q)$ denote the group of projectivities of $\PG(4n-1, q)$ stabilizing $\cQ^+(4n-1, q)$. It follows that
$$
\cL_1 = \left\{ \phi \left( \langle \x \rangle_{q^2} \right) \;\; | \;\; \0 \ne \x \in V(N, q^2), H(\x, \x) = 0 \right\} \subset \cL
$$
is a line--spread of $\cQ^+(4n-1, q)$, see also \cite{D}.

\begin{prop}
There exists a group $G_n \le \PGO^+(4n, q)$ isomorphic to $\frac{\U(2n, q^2)}{\langle - I_{2n} \rangle}$ stabilizing $\cL$.
\end{prop}
\begin{proof}
Let $M = (a_{i j}) \in \U(2n, q^2)$ and hence $H(\x M^t, \x M^t) = H(\x, \x)$. Denote by $\overline{M}$ the unique element of $\GL(4n, q)$ such that $\phi (\x M^t) = \phi(\x) \overline{M}^t$, forall $\x \in V(2n, q^2)$. Let $G_n$ be the group of projectivities of $\PG(4n-1, q)$ associated with the matrices $\{\overline{M} \;\; | \;\; M \in \U(2n, q^2)\}$. Let $g \in G_n$, then $g$ stabilizes $\cL$. Moreover $Q(\phi(\x)\overline{M}^t) = Q(\phi (\x M^t)) = H(\x M^t, \x M^t) = H(\x, \x) = Q(\phi(\x))$ and hence $g$ fixes $\cQ^+(4n-1, q)$. Some calculations show that if $M = (a_{i j})$, where $a_{i j} = b_{i j} + w c_{i j}$, then
$$
\overline{M} = ( A_{i j} ), \mbox{ where }
A_{i j} =
\begin{pmatrix}
b_{i j} & c_{i j} \alpha \\
c_{i j} & b_{i j} + c_{i j} \\
\end{pmatrix}.
$$
It follows that if $\overline{M}$ induces the identity projectivity, then $A_{i j} = \0$, $i \ne j$, and $A_{i i} = \lambda I_{2}$, $1 \le i \le 2n$, for some $\lambda \in \GF(q) \setminus \{0\}$. Hence $M = \lambda I_{2n}$. On the other hand, since $M \in \U(2n, q^2)$, we have that $\lambda = \pm 1$.
\end{proof}

Let $K_n$ be the subgroup of $G_n$ consisting of projectivities of $\PG(4n-1, q)$ associated with the set of matrices $\{\overline{M} \;\; | \;\; M \in \SU(2n, q^2)\}$. We have that $K_n \simeq \frac{\SU(2n, q^2)}{\langle - I_{2n} \rangle}$.

\begin{prop}\label{points}
The groups $K_n$ and $G_n$ act transitively on points of $\cQ^+(4n-1, q)$ and lines of $\cL_1$.
\end{prop}
\begin{proof}
By Witt's theorem the group $\U(2n, q^2)$ acts transitively on $1$--dimensional subspaces of $V(2n, q^2)$ that are totally isotropic with respect to $H$. Hence $G_n$ acts transitively on lines of $\cL_1$. Alternatively, let $\u = (0, \dots, 1) \in V(2n, q^2)$ and consider a matrix $M \in \U(2n, q^2)$. Some straightforward calculations show that $M$ stabilizes $\langle \u \rangle_{q^2}$ if and only if it has the following form
\begin{equation}\label{matrix}
\begin{pmatrix}
  &      & & a_{1} & 0 \\
  & M' & & \vdots & \vdots \\
  &      & & a_{2n-2} & 0 \\
0 & \dots & 0 & a_{2n-1} & 0 \\
a_1'& \dots & a_{2n-2}' & a_{2n-1}' & a_{2n-1}^{-q} \\
\end{pmatrix},
\end{equation}
where $M'^t J_{n-1} M'^q = J_{n-1}$, $\det(M') \ne 0$, $a_1, \dots, a_{2n-2} \in \GF(q^2)$, $a_{2n-1} \in \GF(q^2) \setminus \{0\}$, $(a_{1}', \dots, a_{2n-2}')$ satisfies
$$
(a_{1}', \dots, a_{2n-2}') J_{n-1} M'^q = - \xi^q a_{2n-1}^{-q} (a_1^q, \dots, a_{2n-2}^q),
$$
and $a_{2n-1}'$ is such that $\xi a_1' a_2'^q + \xi^q a_1'^q a_2' + \ldots + \xi a_{2n-1}' a_{2n}'^q + \xi^q a_{2n-1}'^q a_{2n}' = 0$. Then $|Stab_{\U(2n, q^2)}(\langle \u \rangle_{q^2})| = q^{4n-3} (q^2-1) |\U(2n-2, q^2)|$ and $G_n$ is transitive on lines of $\cL_1$. In a similar way $M$ stabilizes $\langle \u \rangle_{q}$ if and only if $M$ has form \eqref{matrix}, where $M'^t J_{n-1} M'^q = J_{n-1}$, $\det(M') \ne 0$, $a_1, \dots, a_{2n-2} \in \GF(q^2)$, $a_{2n-1} \in \GF(q) \setminus \{0\}$, $(a_{1}', \dots, a_{2n-2}')$ satisfies
$$
(a_{1}', \dots, a_{2n-2}') J_{n-1} M'^q = - \xi^q a_{2n-1}^{-q} (a_1^q, \dots, a_{2n-2}^q),
$$
and $a_{2n-1}'$ is such that $\xi a_1' a_2'^q + \xi^q a_1'^q a_2' + \ldots + \xi a_{2n-1}' a_{2n}'^q + \xi^q a_{2n-1}'^q a_{2n}' = 0$. Then $|Stab_{\U(2n, q^2)}(\langle \u \rangle_{q})| = q^{4n-3} (q-1) |\U(2n-2, q^2)|$ and $G_n$ is transitive on points of $\cQ^+(4n-1, q)$.

On the other hand, $M \in \SU(2n, q^2)$ stabilizes $\langle \u \rangle_{q^2}$ if and only if $M$ has form \eqref{matrix}, where $M'^t J_{n-1} M'^q = J_{n-1}$, $\det(M') \ne 0$, $a_1, \dots, a_{2n-2} \in \GF(q^2)$, $a_{2n-1}^{q-1} = \det(M')$, $(a_{1}', \dots, a_{2n-2}')$ satisfies
$$
(a_{1}', \dots, a_{2n-2}') J_{n-1} M'^q = - \xi^q a_{2n-1}^{-q} (a_1^q, \dots, a_{2n-2}^q),
$$
and $a_{2n-1}'$ is such that $\xi a_1' a_2'^q + \xi^q a_1'^q a_2' + \ldots + \xi a_{2n-1}' a_{2n}'^q + \xi^q a_{2n-1}'^q a_{2n}' = 0$. Then $|Stab_{\SU(2n, q^2)}(\langle \u \rangle_{q^2})| = q^{4n-3} (q-1) |\U(2n-2, q^2)|$ and $G_n$ is transitive on lines of $\cL_1$. In the same fashion $M \in \SU(2n, q^2)$ stabilizes $\langle \u \rangle_{q}$ if and only if $M$ has the form \eqref{matrix}, where $M'^t J_{n-1} M'^q = J_{n-1}$, $\det(M') = 1$, $a_1, \dots, a_{2n-2} \in \GF(q^2)$, $a_{2n-1} \in \GF(q) \setminus \{0\}$, $(a_{1}', \dots, a_{2n-2}')$ satisfies
$$
(a_{1}', \dots, a_{2n-2}') J_{n-1} M'^q = - \xi^q a_{2n-1}^{-q} (a_1^q, \dots, a_{2n-2}^q),
$$
and $a_{2n-1}'$ is such that $\xi a_1' a_2'^q + \xi^q a_1'^q a_2' + \ldots + \xi a_{2n-1}' a_{2n}'^q + \xi^q a_{2n-1}'^q a_{2n}' = 0$. Then $|Stab_{\SU(2n, q^2)}(\langle \u \rangle_{q})| = q^{4n-3} (q-1) |\SU(2n-2, q^2)|$ and $G_n$ is transitive on points of $\cQ^+(4n-1, q)$.
\end{proof}

\begin{theorem}\label{hyp}
The group $G_n$ has $n+1$ orbits, say $\cO_{n, i}$, $0 \le i \le n$, on generators of $\cQ^+(4n-1, q)$, where
$$
|\cO_{n, 0}| = q^{n^2-n} \prod_{j = 1}^{n} (q^{2j-1} +1), \quad |\cO_{n, n}| = \prod_{j = 1}^{n} (q^{2j-1} + 1),
$$
$$
|\cO_{n, i}| = q^{(n-i)(n-i-1)} \frac{\prod_{j = n-i+1}^{n} (q^{2j} - 1)}{\prod_{j = 1}^{i} (q^{2j} - 1)} \prod_{j = 1}^{n} (q^{2j-1} +1), \quad 1 \le i \le n-1,
$$
and a member of $\cO_{n, i}$ contains exactly $(q^{2i}-1)/(q^2-1)$ lines of $\cL_1$. Moreover $\cO_{n, i}$ is a $K_n$--orbit if $1 \le i \le n$, whereas $\cO_{n, 0}$ splits into $q+1$ equally sized orbits under the action of $K_n$.
\end{theorem}
\begin{proof}
By induction on $n$. The statement holds true for $n = 1$. Indeed $G_1$ has two orbits on generators of $\cQ^+(3, q)$. Let $\Pi$ be a generator of $\cQ^+(4n-1, q)$ containing at least one line of $\cL_1$. Since $G_n$ is transitive on lines of $\cL_1$, we may assume that $\Pi$ contains the line $\ell = \phi(\langle \u \rangle_{q^2}) \in \cL_1$, where $\u = (0, \dots, 0,1)$. Consider the $(4n-5)$--space $\Lambda$ of $\PG(4n-1, q)$ given by $y_{2n-1} = z_{2n-1} = y_{2n} = z_{2n} = 0$. Then $\Lambda$ meets $\cQ^+(4n-1, q)$ in a $\cQ^+(4n-5, q)$ and $\Pi \cap \Lambda$ is a generator of $\cQ^+(4n-5, q)$. From the proof of Proposition \ref{points}, we have that the stabilizer of $\ell$ in $G_n$, when restricted on $\Lambda$, coincides with $G_{n-1}$. By induction hypothesis the group $G_{n-1}$ has $n$ orbits on generators of $\cQ^+(4n-5, q)$, namely $\cO_{n-1, i}$, $0 \le i \le n-1$. Let $\Pi$ be a member of $\cO_{n, i}$ whenever $\Pi \cap \Lambda$ belongs to $\cO_{n-1, i-1}$, where $1 \le i \le n$. Note that $\Pi$ contains $(q^{2i}-1)/(q^2-1)$ lines of $\cL_1$ if and only if $\Pi \cap \Lambda$ contains $(q^{2i-2}-1)/(q^2-1)$ lines of $\cL_1$. In order to calculate the size of $\cO_{n, i}$, $1 \le i \le n$, set $\cO_{0, 0} = 1$ and observe that
$$
|\cO_{n, i}| =  |\cO_{n-i, 0}| \cdot \mbox{\# of $(i-1)$--spaces of } \cH(2n-1, q^2).
$$
It follows that the number of generators of $\cQ^+(4n-1, q)$ not contained in $\cup_{i = 1}^{n} \cO_{n, i}$ or equivalently containing no line of $\cL_1$ equals $q^{n^2-n} \prod_{j = 1}^{n} (q^{2j-1} +1)$. Let us consider the following vector space over $\GF(q)$: $\{\langle \x \rangle_{q} \;\; | \;\; \0 \ne \x \in V(2n, q^2), \x^q = \x\}$ and assume that $\Pi = \{\phi(\langle \x \rangle_{q}) \;\; | \;\; \0 \ne \x \in V(2n, q^2), \x^q = \x\}$. Then $\Pi$ is the generator of $\cQ^+(4n-1, q)$ given by $z_1 = \ldots = z_{2n} = 0$ and no line of $\cL_1$ is contained in $\Pi$. Let $g \in G_n$ such that $\Pi^g = \Pi$, where $g$ is represented by the matrix $\overline{M}$. Thus $M = (a_{i j}) \in \U(2n, q^2)$ is such that $M$ fixes $\Sigma$. Therefore $a_{i j}^q = a_{i j}$ and $M^t J'_n M = J'_n$, that is $M \in \Sp(2n, q)$. Viceversa, if $M \in \Sp(2n, q)$, then the projectivity of $G_n$ represented by $\overline{M}$ stabilizes $\Pi$. Hence $Stab_{G_n}(\Pi) = \Sp(2n, q)/\langle -I_{2n} \rangle$. It follows that $|\Pi^{G_n}| = q^{n^2-n} \prod_{j = 1}^{n} (q^{2j-1} +1)$ and $G_n$ is transitive on generators of $\cQ^+(4n-1, q)$ containing no line of $\cL_1$.

In order to study the action of the group $K_n$ on generators of $\cQ^+(4n-1, q)$, note that, from the proof of Proposition \ref{points}, the stabilizer of $\ell$ in $K_n$, when restricted on $\Lambda$, coincides with $G_{n-1}$. Hence an analogous inductive argument to that used in the previous case gives that $\cO_{n, i}$, $1 \le i \le n$, is a $K_n$--orbit. Let $\Pi$ be a generator of $\cQ^+(4n-1, q)$ containing no line of $\cL_1$, i.e., $\Pi \in \cO_{n, 0}$. The lines of $\cL_1$ meeting $\Pi$ in one point are the images under $\phi$ of the points of a Baer subgeometry $\Sigma$ embedded in $\cH(2n-1, q^2)$. Hence if $\overline{M} \in Stab_{K_n}(\Pi)$, then $M \in Stab_{\SU(2n, q^2)}(\Sigma)$, that is conjugate in $\U(2n, q^2)$ to $\langle \overline{\Delta}, \Sp(2n, q) \rangle$ by Proposition \ref{stabilizer2}. On the other hand if $\overline{M}$ is conjugate to an element of $\langle \overline{\Delta}, \Sp(2n, q) \rangle$ and it induces a projectivity that fixes $\Pi$, then $M \in \Sp(2n, q)$. Therefore $|Stab_{K_n}(\Pi)| = |\Sp(2n, q)/\langle -I_{2n} \rangle|$ and $|\Pi^{K_n}| = q^{n^2-n} \prod_{j = 2}^{n} (q^{2j-1} +1)$.
\end{proof}

\subsubsection{The elliptic case}

Assume that $N = 2n+1$. Consider the following non--degenerate Hermitian form $H$ on $V(2n + 1, q^2)$
$$
H(\x, \x') = \x \tilde{J_n} ({\x'}^q)^t = \begin{pmatrix} x_1, \dots, x_{2n+1} \end{pmatrix}
\tilde{J_n}
\begin{pmatrix} x_1'^q, \dots, x_{2n+1}'^q \end{pmatrix}^t ,
$$
where $\tilde{J_n}$ is given by
\begin{equation}\label{hermitian_matrix_1}
\begin{pmatrix}
0 & \xi & \cdots & 0 & 0 & 0 \\
\xi^q & 0 & \cdots & 0 & 0 & 0 \\
\vdots & \vdots & \ddots & \vdots & \vdots & \vdots \\
0 & 0 & \cdots & 0 & \xi & 0 \\
0 & 0 & \cdots & \xi^q & 0 & 0 \\
0 & 0 & \cdots & 0 & 0 & - \xi^2
\end{pmatrix}.
\end{equation}
Consider the following groups of unitary isometries
\begin{align*}
\U(2n+1, q^2) = \{M \in \GL(2n+1, q^2) \;\; | \;\; M^t \tilde{J_n} M^q = \tilde{J_n}\}, \\
\SU(2n+1, q^2) = \{M \in \SL(2n+1, q^2) \;\; | \;\; M^t \tilde{J_n} M^q = \tilde{J_n}\}, \\
\end{align*}
We have that
\begin{align*}
 H(\x, \x) & = \xi (x_1 x_2^q - x_1^q x_2 + \ldots + x_{2n-1} x_{2n}^q - x_{2n-1}^q x_{2n} - \xi x_{2n+1}^{q+1}) \\
  & = - \xi^2 (y_1 z_2 - z_1 y_2 + \ldots + y_{2n-1} z_{2n} - z_{2n-1} y_{2n} + y_{2n+1}^2 + y_{2n+1} z_{2n+1} - \alpha z_{2n+1}^2) = Q(\y).
\end{align*}
Note that $Q$ is a non--degenerate quadratic form on $V(4n+2, q)$ of elliptic type. Let $\cH(2n, q^2)$ be the Hermitian variety of $\PG(2n, q^2)$ determined by $H$. Denote by $\cQ^-(4n+1, q)$ the hyperbolic quadric of $\PG(4n+1, q)$ defined by $Q$ and let $\PGO^-(4n+2, q)$ denote the group of projectivities of $\PG(4n+1, q)$ stabilizing $\cQ^-(4n+1, q)$. It follows that
$$
\cL_1 = \left\{ \phi \left( \langle \x \rangle_{q^2} \right) \;\; | \;\; \0 \ne \x \in V(N, q^2), H(\x, \x) = 0 \right\} \subset \cL
$$
is a line--spread of $\cQ^-(4n+1, q)$, see also \cite{D}. Similarly to the previous case, the following holds.

\begin{prop}
There exists a group $\tilde{G}_n \le \PGO^-(4n+2, q)$ isomorphic to $\frac{\U(2n+1, q^2)}{\langle - I_{2n+1} \rangle}$ stabilizing $\cL$.
\end{prop}

Let $\tilde{K}_n$ be the subgroup of $\tilde{G}_n$ consisting of projectivities of $\PG(4n+1, q)$ associated with the matrices $\{\overline{M} \;\; | \;\; M \in \SU(2n+1, q^2)\}$. We have that $\tilde{K}_n \simeq \frac{\SU(2n+1, q^2)}{\langle - I_{2n+1} \rangle}$.

\begin{prop}\label{points_1}
The groups $\tilde{K}_n$ and $\tilde{G}_n$ act transitively on lines of $\cL \setminus \cL_1$ and for $\ell \in \cL \setminus \cL_1$, $Stab_{\tilde{K}_n}(\ell) \simeq G_n$ and $Stab_{\tilde{G}_{n}}(\ell) \simeq (q+1) \times G_n$.
\end{prop}
\begin{proof}
Let $\u = (0, \dots, 1) \in V(2n+1, q^2)$ and consider a matrix $M \in \U(2n+1, q^2)$. Some straightforward calculations show that $M$ stabilizes $\langle \u \rangle_{q^2}$ if and only if it has the form
\begin{equation}\label{matrix_1}
\begin{pmatrix}
   &          & & 0 \\
   &  M'    & & \vdots \\
   &          & & 0 \\
0 & \dots & 0 & a \\
\end{pmatrix},
\end{equation}
where $M'^t J_{n} M'^q = J_{n}$, $\det(M') \ne 0$, $a \in \GF(q^2)$, $a^{q+1} = 1$. Then $|Stab_{\U(2n+1, q^2)}(\langle \u \rangle_{q^2})| = (q+1) |\U(2n, q^2)|$. Hence $\tilde{G}_n$ is transitive on lines of $\cL \setminus \cL_1$, since $|\U(2n+1, q^2)/Stab_{\U(2n+1, q^2)}(\langle \u \rangle_{q^2})| = |\cL \setminus \cL_1|$.

On the other hand, $M \in \SU(2n, q^2)$ stabilizes $\langle \u \rangle_{q^2}$ if and only if $M$ is of type \eqref{matrix_1}, where $M'^t J_{n} M'^q = J_{n}$, $\det(M') \ne 0$ and $a = \det(M')^{-1}$. Then $|Stab_{\SU(2n+1, q^2)}(\langle \u \rangle_{q^2})| = |\U(2n, q^2)|$. Therefore $|\SU(2n+1, q^2)/Stab_{\SU(2n+1, q^2)}(\langle \u \rangle_{q^2})| = |\cL \setminus \cL_1|$ and $\tilde{K}_n$ is transitive on lines of $\cL \setminus \cL_1$.
\end{proof}

\begin{theorem}
The group $\tilde{G}_n$ has $n+1$ orbits, say $\tilde{\cO}_{n, i}$, $0 \le i \le n$, on generators of $\cQ^-(4n+1, q)$, where
$$
|\tilde{\cO}_{n, 0}| = q^{n^2+n} \prod_{j = 2}^{n+1} (q^{2j-1} +1), \quad |\tilde{\cO}_{n, n}| = \prod_{j = 2}^{n+1} (q^{2j-1} + 1),
$$
$$
|\tilde{\cO}_{n, i}| = q^{(n-i)(n-i+1)} \frac{\prod_{j = n-i+1}^{n} (q^{2j} - 1)}{\prod_{j = 1}^{i} (q^{2j} - 1)} \prod_{j = 2}^{n+1} (q^{2j-1} +1), \quad 1 \le i \le n-1,
$$
and a member of $\tilde{\cO}_{n, i}$ contains exactly $(q^{2i}-1)/(q^2-1)$ lines of $\cL_1$. 
\end{theorem}
\begin{proof}
Let $\ell$ be a line of $\cL \setminus \cL_1$. Then $\ell^\perp \cap \cQ^-(4n+1, q) = \cQ^+(4n-1, q)$ and $Stab_{\tilde{G}_{n}}(\ell) \simeq (q+1) \times G_n$. Hence, from Theorem \ref{hyp}, the group $Stab_{\tilde{G}_n}(\ell)$ has $n+1$ orbits, say $\cO_{n, i}$, $0 \le i \le n$, on generators of $\cQ^-(4n+1, q)$ contained in $\ell^\perp$. Moreover $\tilde{G}_n$ is transitive on lines of $\cL \setminus \cL_1$. The result follows by counting in two ways the couples $(\ell, g)$, where $\ell \in \cL \setminus \cL_1$, $g$ is a generator of $\cQ^-(4n+1, q)$ containing exactly $(q^{2i}-1)/(q^2-1)$ lines of $\cL_1$ and $\ell \in g^\perp$. Indeed, on one hand we have that this number equals $|\cL \setminus \cL_1| \cdot |\cO_{n, i}| = \frac{q^{2n} (q^{2n + 1} + 1)}{q+1} \cdot |\cO_{n, i}|$, whereas on the other hand it coincides with $\left(\frac{(q^2-1)(q^{2i} - 1)}{q^2-1} + 1\right) \cdot |\tilde{\cO}_{n, i}| = q^{2i} \cdot |\tilde{\cO}_{n, i}|$.
\end{proof}


\end{document}